\documentclass[oneside, 11pt]{amsart} 

\usepackage{template}
\usepackage{preamble}

\title{Higher structures in rational homotopy theory}
\date{\today}

\author{Alexander Berglund}
\address{Department of Mathematics, Stockholm University, 106 91 Stockholm, Sweden}
\email{\href{mailto:alexb@math.su.se}{alexb@math.su.se}}

\author{Robin Stoll}
\address{Department of Pure Mathematics and Mathematical Statistics, University of Cambridge, CB3 0WB Cambridge, United Kingdom}
\email{\href{mailto:rs2348@cam.ac.uk}{rs2348@cam.ac.uk}}



\begin{document}
\begin{abstract}
These notes are based on a series of three lectures given (online) by the first named author at the workshop \emph{Higher Structures and Operadic Calculus} at CRM Barcelona in June 2021.
The aim is to give a concise introduction to rational homotopy theory through the lens of higher structures.

The rational homotopy type of a simply connected space of finite type is modeled by a $\Cinfty$-algebra structure on the rational cohomology groups, or alternatively an $\Linfty$-algebra structure on the rational homotopy groups. The first lecture is devoted to explaining these models and their relation to the classical models of Quillen and Sullivan.

The second lecture discusses the relation between Koszul algebras, formality and coformality. The main result is that a space is formal if and only if the rational homotopy $\Linfty$-algebra is Koszul and, dually, a space is coformal if and only if the cohomology $\Cinfty$-algebra is Koszul. For spaces that are both formal and coformal, this collapses to classical Koszul duality between Lie and commutative algebras.

In the third lecture, we discuss certain higher structure in the rational homotopy theory of automorphisms of high dimensional manifolds, discovered by Berglund--Madsen. The higher structure in question is Kontsevich's Lie graph complex and variants of it.


  

\end{abstract}

\maketitle

\tableofcontents

\section*{Introduction}

The purpose of these lecture notes is to give an introduction to rational homotopy theory from the perspective of higher structures. For readers already acquainted with rational homotopy theory as presented in standard textbooks, these lecture notes may serve as an introduction to the language of algebraic operads, $\Cinfty$- and $\Linfty$-algebras, Koszul duality, etc., by seeing how it is used in a familiar context. Conversely, for readers with a background in the theory of algebraic operads, these lecture notes may serve as a concise introduction to rational homotopy theory. For readers with background in none of the theories, we hope the text can function as an introduction to both. Needless to say, it is only possible to cover a small amount of material in three lectures. We do not give full proofs of all statements but we have tried to provide references to more complete accounts throughout.
\vskip8pt
The goal of the first lecture is to explain, and outline the proof of, the result that the rational homotopy type of a simply connected space of finite type is faithfully encoded by a $\Cinfty$-algebra structure on the rational cohomology groups, or alternatively by an $\Linfty$-algebra structure on the rational homotopy groups.
We begin by reviewing the definition of $\Cinfty$- and $\Linfty$-algebras, and then explain how the theory of algebraic operads provides a unified framework for developing the basic theory of these higher structures. We then present a short proof of the homotopy transfer theorem, which is one of the main technical tools. Next, we review the parts of Sullivan's and Quillen's theories that are necessary for proving the main result and we discuss the relation between the $\Cinfty$- and $\Linfty$-algebra models and the minimal models of classical rational homotopy theory.
\vskip8pt
The second lecture discusses the relation between Koszul algebras and the notions of formality and coformality. We review the definition of Koszul commutative and Lie algebras, and the recent generalization to Koszul $\Cinfty$- and $\Linfty$-infinity algebras. We emphasize that we view Koszul algebras as a calculational tool that, in favorable situations, allows one to, for example, quickly compute the rational homotopy groups of a space from a presentation of the rational cohomology ring, or quickly decide whether a given space is (co)formal or not. We give a number of examples to illustrate this point.
\vskip8pt
In the third lecture, we discuss certain higher structure in the rational cohomology of classifying spaces of automorphisms of high dimensional manifolds, discovered by Berglund--Madsen. The higher structure in question is Kontsevich's Lie graph complex and variants of it. We begin by reviewing some general results on the rational homotopy theory of spaces of self--homotopy equivalences of simply connected finite CW-complexes. Then we review the definition of modular operads and graph complexes. Finally, we sketch the results of Berglund--Madsen.
\vskip8pt
In an appendix, we discuss how the $\Cinfty$- and $\Linfty$-models relate to Massey operations and higher order Whitehead products, respectively. We also review the nerve of an $\Linfty$-algebra and related constructions, and we summarize some facts about rational homotopy theory of non-nilpotent spaces.


\subsection*{Conventions}
We work over the field of rational numbers $\QQ$, except when stated otherwise. In particular ``vector space'' will mean vector space over $\QQ$, tensor products are over $\QQ$, and so on. 

We use homological grading convention, so differentials in chain complexes have degree $-1$, and chain complexes are unbounded unless stated otherwise. Thus, a chain complex is a collection $V = \{V_i\}_{i\in \ZZ}$ of vector spaces together with a linear map $d\colon V \to V$ of degree $-1$ such that $d^2 = 0$. We let $|v|$ denote the degree of an element $v\in V$, i.e.\ $|v|=i$ means $v\in V_i$.
By using the convention $V^i = V_{-i}$, chain complexes can be considered as cochain complexes and vice versa, that is, subscripts refer to the homological degree and superscripts to the negated homological (i.e.\ cohomological) degree.

We use $d$ as generic notation for the differential of a chain complex when it is clear from the context what is meant. For example, the differential of the tensor product $V\otimes W$ of two chain complexes can be written as
$$d(v\otimes w) = d(v) \otimes w  +(-1)^{|v|}v\otimes d(w)$$
without risk of confusion. We use $\partial$ as generic notation for differentials in Hom-complexes $\Hom(V,W)$, so that
$$\partial(f) = d\circ f - (-1)^{|f|} f\circ d$$
for $f\in \Hom(V,W)$.

The dual of a chain complex $V$ is denoted $\dual{V} = \Hom(V,\QQ)$. The suspension $\shift V$ is defined by $(\shift V)_i = V_{i-1}$ with differential $d(\shift v) = -\shift d(v)$ for $v\in V$.


\section{Lecture 1: \texorpdfstring{$\Linfty$}{L∞}- and \texorpdfstring{$\Cinfty$}{C∞}-models for spaces}

The primary task of rational homotopy theory is to classify topological spaces up to rational equivalence.

\begin{definition}
A map $f \colon X \to Y$ between topological spaces is called a \emph{rational equivalence} if the induced map in rational homology,
$$f_*\colon \Ho{*}(X;\QQ) \longto \Ho{*}(Y;\QQ),$$
is an isomorphism.
  
Rational equivalences need not have inverses, but rational equivalence generates an equivalence relation $\rateq$ and we say that two spaces $X$ and $Y$ are \emph{rationally equivalent}, or have the same \emph{rational homotopy type}, if $X \rateq Y$.
\end{definition}

We will mainly focus on the simply connected case. If $X$ and $Y$ are simply connected, then $f \colon X \to Y$ is a rational equivalence if and only if the induced map on rational homotopy groups,
$$f_*\colon \htpygrp{*}(X)\otimes \QQ \longto \htpygrp{*}(Y)\otimes \QQ,$$
is an isomorphism (see e.g.\ \cite[Theorem 8.6]{FelixHalperinThomas01}).
Evidently, the rational cohomology and homotopy groups are invariants of the rational homotopy type of a simply connected space. Somewhat surprisingly, a complete solution to the classification problem can be obtained by adding certain higher structure to these invariants. The aim of the first lecture is to explain the statements and proofs of the following two theorems.

\begin{theorem} \label{thm:C-infinity model}
  For each simply connected space $X$ of finite $\QQ$-type, there exists a $\Cinfty$-algebra structure on the rational cohomology groups $\Coho{*}(X;\QQ)$ that provides a complete invariant of the rational homotopy type, in the sense that $X$ and $Y$ are rationally equivalent if and only if the $\Cinfty$-structures on $\Coho{*}(X;\QQ)$ and $\Coho{*}(Y;\QQ)$ are $\Cinfty$-isomorphic.
\end{theorem}

\begin{theorem}
\label{thm:L-infinity model}
For each simply connected space $X$ of finite $\QQ$-type, there exists an $\Linfty$-algebra structure on the shifted rational homotopy groups $\htpygrp{* + 1}(X) \tensor \QQ$ that provides a complete invariant of the rational homotopy type, in the sense that $X$ and $Y$ are rationally equivalent if and only if the $\Linfty$-structures on  $\htpygrp{* + 1}(X) \tensor \QQ$ and $\htpygrp{* + 1}(Y) \tensor \QQ$ are $\Linfty$-isomorphic.
\end{theorem}

Here, \emph{finite $\QQ$-type} means that $\Ho{k}(X;\QQ)$ is finite dimensional for each $k$.

On one hand, both theorems are mere reformulations of results that have been known since the early days of rational homotopy theory; we will discuss the relation to Quillen's and Sullivan's theories in \cref{sec:nothing new}.

On the other hand, the perspective of higher structures has certain advantages and it has played a considerable role in recent developments of rational homotopy theory and its applications. The language of $\Cinfty$- and $\Linfty$-algebras does not only shed new light on existing results, it also leads to new results that would be cumbersome to state and prove using only the classical theory.

Theorem \ref{thm:C-infinity model} has been stated and proved in the language of $\Cinfty$-algebras by Kadeishvili \cite{Kadeishvili09}.
Statements along the lines of \cref{thm:L-infinity model} can be found in various sources, for example \cite[\S 1]{LadaMarkl95}, but a treatment parallel to \cite{Kadeishvili09} does not seem to have appeared in the literature.




\subsection{\texorpdfstring{$\Ainfty$}{A∞}-algebras}
A $\Cinfty$-algebra is a special kind of $\Ainfty$-algebra so we begin with the latter. 
$\Ainfty$-algebras were introduced by Stasheff \cite[\S2]{Stasheff63}. We will here only review a small part of the theory. For a less condensed introduction, we recommend \cite{Keller01}. A detailed comprehensive treatment can be found in \cite{LefevreHasegawa03}.

\begin{definition}
An \emph{$\Ainfty$-algebra structure} on a chain complex $(A, d)$ consists of a family of operations
  \[ m_n \colon A^{\tensor n}  \longto  A,\quad n = 2, 3, \ldots,\]
  of degree $n - 2$ such that the equation
  \[ \sum_{r + s + t = n} (-1)^{r + st} m_{r + 1 + t} \after ({\id^{\tensor r}} \tensor m_s \tensor {\id^{\tensor t}}) = 0 \]
  holds for every $n \ge 2$, where the sum is over all positive integers $r,s,t$ such that $r+s+t =n$, and where we set $m_1 = d$.
\end{definition}
Note that for $n = 2$ the above equation says
\[ d \after m_2 - m_2 \after (d \tensor {\id} + {\id} \tensor d) = 0 .\]
In other words, $m_2$ is a chain map $A^{\tensor 2} \to A$. If we let $\del$ denote the differential of the Hom-complex $\Hom(A^{\tensor k}, A)$, this equation can be rewritten as
$$\del(m_2) = 0.$$ 
Similarly, for $n=3$, the equation can be written as
\begin{equation} \label{eq:m_3}
\del(m_3) = m_2 \after ({\id} \tensor m_2) - m_2 \after (m_2 \tensor {\id}).
\end{equation}
This means that $m_2$ is associative up to homotopy, and that $m_3$ is a prescribed chain homotopy between the ``associator'' and the zero map.
One can also think of \eqref{eq:m_3} as the first step in a resolution of the associativity relation. From this perspective, the higher operations $m_4,m_5,\ldots$ take care of the syzygies in this resolution.

There is an alternative, more compact, definition of $\Ainfty$-algebras. Part of this is contained in \cite[(2.4)]{Stasheff63}; the form we state can be found in \cite[Lemma 1.2.2.1]{LefevreHasegawa03}.
We write $\coTA(V)$ for the \emph{tensor coalgebra} on a chain complex $V$, defined by
$$\coTA(V) = \Dirsum_{n \ge 0} V^{\tensor n}.$$
Using the ``bar notation'' $[v_1|\ldots|v_n] = v_1\otimes \ldots \otimes v_n$ for $v_i\in V$, the coproduct $\Delta\colon \coTA(V) \to \coTA(V)\otimes \coTA(V)$ is given by
$$\Delta[v_1|\ldots|v_n] = \sum_{i=0}^n [v_1|\ldots|v_i]\otimes [v_{i+1}|\ldots | v_n],$$
and the counit $\eta\colon \QQ \to \coTA(V)$ by the inclusion of $V^{\tensor 0} = \QQ$. The direct sum decomposition yields an additional grading on $\coTA(V)$, which we call the \emph{weight grading} and whose $n$-th graded piece we denote by $\weight n {\coTA (V)}$. We let $\weight {<n} {\coTA (V)}$ denote the sum of the components of weight $<n$. We let $\redcoTA(V)$ denote the sum of components in positive weight.

\begin{proposition} \label{lemma:Ainfty_equiv_def}
  Let $A$ be a chain complex.
  There is a bijection between $\Ainfty$-algebra structures on $A$ and degree $-1$ self-maps $b$ of $\coTA(\shift A)$ such that the following conditions are fulfilled:
  \begin{itemize}
    \item $b$ is a coderivation, i.e.\ $\Delta b = (b\otimes \id + \id \otimes b)\Delta$.
    \item $b$ is a perturbation of the differential $d$ on $\coTA(\shift A)$, i.e.\ $(d + b)^2 = 0$.
    \item $b$ decreases weight, i.e.\ $b \big( \weight n {\coTA (\shift A)} \big) \subseteq \weight {<n} {\coTA (\shift A)}$ for all $n$.
  \end{itemize}
\end{proposition}

\begin{proof}[Proof sketch]
  One notes that a coderivation $b \colon \coTA(\shift A) \to \coTA(\shift A)$ is uniquely determined by its linear part, i.e.\ the composition
  \[ \Dirsum_n (\shift A)^{\tensor n} = \coTA(\shift A) \xlongto{b} \coTA(\shift A) \xlongto{\pr} \shift A . \]
  The components of this map correspond (up to shifts and signs) to the operations $m_n$.
  Under this correspondence, the $\Ainfty$-relations are equivalent to $(d + b)^2 = 0$.
\end{proof}

A differential graded associative algebra is the same thing as an $\Ainfty$-algebra $A$ such that $m_n = 0$ for all $n\geq 3$. When $A$ is a dg associative algebra, the dg coalgebra $\big(\coTA(\shift A),d+b\big)$ agrees with the classical bar construction $\Bar A$ (as defined in e.g.\ \cite[Chapter 19]{FelixHalperinThomas01}). This justifies the following definition.

\begin{definition}
  Let $A$ be an $\Ainfty$-algebra.
  The \emph{bar construction} of $A$ is the dg coalgebra
  $$\Bar A \defeq \bigl( \coTA(\shift A), d + b \bigr),$$
  where $d + b$ is as in \Cref{lemma:Ainfty_equiv_def}.
\end{definition}

\subsection{\texorpdfstring{$\Cinfty$}{C∞}-algebras}
$\Cinfty$-algebras go back at least to Kadeishvili \cite{Kadeishvili88}, who calls them ``commutative $A(\infty)$-algebras''. They are called ``$\Cinfty$-algebras'' by Getzler--Jones \cite[\S 5.3]{GetzlerJones94}.
To define them we need the following auxiliary notion.

\begin{definition}
  A \emph{$(p,q)$-shuffle} is a permutation $\sigma \in \Symm{p + q}$ such that
  \[ \sigma(1) < \dots < \sigma(p) \qquad \text{and} \qquad \sigma(p+1) < \dots < \sigma(p + q) \]
  hold. We denote the set of $(p,q)$-shuffles by $\Shuffles p q$ and we let $\tau_{p,q}$ denote the distinguished element in the group algebra $\QQ \Symm{p+q}$ given by
  $$\tau_{p,q} = \sum_{\sigma \in \Shuffles p q} \sgn(\sigma) \sigma.$$
\end{definition}

The name ``shuffle'' is motivated by riffle shuffles of a deck of cards: it is cut into two parts which are subsequently mixed together without changing the order of the cards in either of the two parts.
Note however that there is some inconsistency regarding the use of the term; what we call a shuffle is called an ``unshuffle'' in \cite{LadaMarkl95}, but our usage agrees with that of e.g.\ \cite{GetzlerJones94} and \cite{LodayVallette12}.

\begin{definition}
  A \emph{$\Cinfty$-algebra} is an $\Ainfty$-algebra $A$ whose operations $m_n$ fulfill
     \begin{equation} \label{eq:cinfty axiom}
     m_{p + q} \after \tau_{p,q} = 0
     \end{equation}
  for all $p, q \geq 1$.
  Here $\tau_{p,q}$ acts on $A^{\tensor p + q}$ from the left by permuting the tensor factors.
\end{definition}

Note that for $p = q = 1$ the above relation reads as
\[ m_2 = m_2 \after \tau \]
where $\tau$ is the non-trivial element of $\Symm 2$.
In particular $m_2$ equips $A$ with the structure of a graded commutative (but \emph{not} necessarily associative) binary operation and it descends to a graded commutative algebra structure on the homology of $A$. The higher operations $m_n$ may be thought of as ``resolving the associativity relation'' of the commutative operation $m_2$, and \eqref{eq:cinfty axiom} may be thought of as the appropriate commutativity constraint on these.

\Cref{lemma:Ainfty_equiv_def} associates to an $\Ainfty$-algebra $A$ a differential $d + b$ on $\Bar A$.
If $A$ is a $\Cinfty$-algebra, this differential is a derivation with respect to the shuffle product $\shuffle$ of $\Bar A = \coTA (\shift A)$ (cf.\ \cite[Proposition 5.5]{GetzlerJones94}).
In particular $d + b$ descends to a differential on the indecomposables with respect to the shuffle product,
\[ \redcoTA (\shift A) / (\redcoTA (\shift A) \shuffle \redcoTA (\shift A)).\]
The latter is isomorphic to the \emph{cofree Lie coalgebra} $\cofreeLie (\shift A)$ (see e.g.\ \cite[Theorem 1.3.6]{LodayVallette12}).
This construction can also be reversed, yielding the following analog of \Cref{lemma:Ainfty_equiv_def} (cf.\ \cite[Proposition 13.1.6]{LodayVallette12}).

\begin{proposition} \label{lemma:Cinfty_equiv_def}
  Let $A$ be a chain complex. There is a bijection between $\Cinfty$-algebra structures on $A$ and degree $-1$ self-maps $b$ of $\cofreeLie(\shift A)$ such that the following conditions are fulfilled:
  \begin{itemize}
    \item $b$ is a coderivation.
    \item $b$ is a perturbation of the differential $d$ on $\cofreeLie(\shift A)$, i.e.\ $(d + b)^2 = 0$.
    \item $b$ decreases weight, i.e.\ $b \big( \weight w {\cofreeLie (\shift A)} \big) \subseteq \weight {<w} {\cofreeLie (\shift A)}$.
  \end{itemize}
\end{proposition}

\begin{definition}
  Let $A$ be a $\Cinfty$-algebra. We define the dg Lie coalgebra
  \[ \Harr * (A)  \defeq  \bigl( \cofreeLie (\shift A), d + b \bigr) \]
  where $d + b$ is as in \Cref{lemma:Cinfty_equiv_def}.
  We write $\coHarr * (A) \defeq \dual{\Harr * (A)}$ for the dual dg Lie algebra.
\end{definition}

\begin{remark}
A commutative dg algebra (cdga) is the same thing as a $\Cinfty$-algebra $A$ such that $m_n = 0$ for $n\geq 3$. In this case, $\Harr * (A)$ agrees with the classical Harrison complex of \cite{Harrison62} (see also \cite[\S 13.1.7]{LodayVallette12}). The dual $\coHarr * (A)$ agrees with the construction studied in \cite[I.1.(7)]{Tanre83}.
\end{remark}

\subsection{\texorpdfstring{$\Linfty$}{L∞}-algebras}
$\Linfty$-algebras are a Lie algebra analog of $\Ainfty$-algebras. One of the earliest sources is Lada--Stasheff \cite{LadaStasheff93}, where they are called ``strongly homotopy Lie algebras''.

\begin{definition}
  An \emph{$\Linfty$-algebra} structure on a chain complex $(L, d)$ is a family of operations
  \[ l_n \colon L^{\tensor n}  \longto  L, \quad n = 2, 3, \dots \]
  of degree $n - 2$ such that each $l_n$ is anti-symmetric (i.e.\ invariant under the sign action of $\Symm n$) and such that
  \[ \del(l_{n}) = \sum_{\substack{p, q \ge 2 \\ p + q = n + 1}} \sum_{\sigma \in \Shuffles q {p - 1}} \sgn(\sigma) (-1)^{p(q-1)} l_p \after (l_q \tensor \id^{\tensor p - 1}) \after \inv{\sigma} \]
  holds.\footnote{There are different sign conventions in the literature. The convention used here agrees with \cite{LadaStasheff93} but differs from \cite{LodayVallette12}.}
  Here $\del$ is the differential of the chain complex $\Hom(L^{\tensor n}, L)$.
\end{definition}

Note that for $n = 2$ the relation above reads as
\[ \del(l_2) = 0 \]
which is equivalent to $l_2$ being a map of chain complexes.
Similarly, writing $[\blank, \blank] \defeq l_2(\blank, \blank)$, we obtain for $n = 3$ that $\del(l_3)(\alpha_1, \alpha_2, \alpha_3)$ is equal to
\[ [[\alpha_1, \alpha_2], \alpha_3] - (-1)^{\deg{\alpha_2} \deg{\alpha_3}} [[\alpha_1, \alpha_3], \alpha_2] + (-1)^{\deg{\alpha_1} (\deg{\alpha_2} + \deg{\alpha_3})} [[\alpha_2, \alpha_3], \alpha_1] \]
so that $l_3$ provides a witness for the fact that $l_2$ fulfills the graded Jacobi relation after taking homology.
In particular $l_2$ equips the homology of $L$ with a Lie algebra structure.
This means that we can think of the operations $l_n$ as ``resolving the Jacobi relation''.

Again there is an alternative compact definition of $\Linfty$-algebras (cf.\ \cite[\S 3]{LadaStasheff93}).
We write $\coSA(V)$ for the \emph{symmetric coalgebra} on $V$, that is
$$\coSA(V) \defeq \Dirsum_{n \ge 0} \coinv {(V^{\tensor n})} {\Symm n}$$
equipped with the canonical structure of a counital cocommutative coassociative coalgebra.
The direct sum decomposition yields an additional grading on $\coSA(V)$, which we call the \emph{weight grading} and whose $n$-th graded piece we denote by $\weight n {\coSA (V)}$.

\begin{proposition} \label{lemma:Linfty_equiv_def}
Let $L$ be a chain complex. There is a bijection between $\Linfty$-algebra structures on $L$ and degree $-1$ self-maps $b$ of the cocommutative coalgebra $\coSA(\shift L)$ such that the following conditions are fulfilled:
  \begin{itemize}
    \item $b$ is a coderivation.
    \item $b$ is a perturbation of the differential $d$ on $\coSA(\shift L)$, i.e.\ $(d+b)^2 = 0$.
    \item $b$ decreases weight, i.e.\ $b \big( \weight w {\coSA (\shift L)} \big) \subseteq \weight {<w} {\coSA (\shift L)}$.
  \end{itemize}
\end{proposition}

Dg Lie algebras may be identified with $\Linfty$-algebras $L$ such that $l_n = 0$ for $n \ge 3$. For such $L$, the dg coalgebra $\big(\coSA(\shift L),d+b\big)$ agrees with Quillen's generalization of the Chevalley--Eilenberg complex (as defined in \cite[Chapter 22(b)]{FelixHalperinThomas01}). This justifies the following definition.

\begin{definition} \label{def:CE}
  The \emph{Chevalley--Eilenberg complex} of an $\Linfty$-algebra $L$ is the cocommutative dg coalgebra
  $$\CE{*}(L) \defeq \bigl( \coSA(\shift L), d + b \bigr),$$
  where $d + b$ is as in \Cref{lemma:Linfty_equiv_def}.
  We also write $\coCE{*}(L) \defeq \dual{\CE{*}(L)}$ for the dual commutative dg algebra.
\end{definition}

\subsection{Unified theory of \texorpdfstring{$\infty$}{∞}-algebras via algebraic operads}
The reader will have noticed certain patterns in the above discussion of different types of $\infty$-algebras. The theory of Koszul duality for operads, going back to \cite{GinzburgKapranov94,GetzlerJones94}, provides a framework for a unified treatment. A comprehensive introduction can be found in \cite{LodayVallette12} and in what follows we will use notation and terminology from this source. However, a reader who is unfamiliar with this theory and mainly interested the examples of $\Ainfty$-, $\Cinfty$- or $\Linfty$-algebras can in principle read the remainder of the section simplistically by interpreting ``$\operad P_\infty$'' as a placeholder for $\Ainfty$, $\Cinfty$ or $\Linfty$ --- the main point is that the argument follows the same pattern in each case.


For the unified treatment, we fix a dg operad $\operad P$ and a cofibrant resolution
$$\operad P_\infty \xlongrightarrow{\eq} \operad P.$$
For convenience, we assume that $\operad P_\infty$ is of the form $\Omega \operad C$, i.e.\ that it is the cobar construction of some dg cooperad $\operad C$.
Not every resolution is of this form, but resolutions of this form always exist. In general, one can take $\operad C$ to be the bar construction $\Bar \operad P$.
If $\operad P$ is a Koszul operad, a smaller (in fact minimal) resolution is obtained by taking $\operad C$ to be the Koszul dual cooperad $\Koszuldual{\operad P}$.

The main examples are the operads
$$\Ass, \quad \Com, \quad \Lie,$$
governing associative, commutative and Lie algebras, respectively. These operads are Koszul and their Koszul dual cooperads may be identified with $$\Koszuldual{\Ass} = \dual{(\opshift \Ass)}, \quad \Koszuldual{\Com} = \dual{(\opshift \Lie)}, \quad \Koszuldual{\Lie} = \dual{(\opshift \Com)},$$
respectively, where $\opshift$ denotes the operadic suspension and $(-)^\vee$ denotes linear dual.
The resulting cofibrant resolutions are the dg operads
$$\Ass_\infty = \Cobar \dual{(\opshift \Ass)},\quad \Com_\infty = \Cobar \dual{(\opshift \Lie)}, \quad \Lie_\infty = \Cobar \dual{(\opshift \Com)},$$
that govern $\Ainfty$-, $\Cinfty$-, and $\Linfty$-algebras, respectively (cf.\ \cite[\S 10.1.5 f.\, \S 13.1.8]{LodayVallette12}).

We will now discuss a common generalization of the constructions $\coTA$, $\cofreeLie$, and $\coSA$. The \emph{cofree conilpotent} $\operad C$-coalgebra on a chain complex $A$ is the chain complex
\[ \Schur {\operad C} A = \Dirsum_{n = 0}^{\infty} \operad C(n) \tensor[\Symm n] A^{\tensor n}. \]
The cooperad structure of $\operad C$ induces a natural $\operad C$-coalgebra structure on $\Schur {\operad C} A$.
We call $\weight w {\Schur {\operad C} A}  \defeq  \operad C(w) \tensor[\Symm w] A^{\tensor w}  \subseteq  \Schur {\operad C} A$ the \emph{weight} $w$ part of $\Schur {\operad C} A$. The following can be found in \cite[Proposition 2.15]{GetzlerJones94}.

\begin{proposition} \label{lemma:Pinfty_equiv_def}
  Let $A$ be a chain complex.
  There is a bijection between $\operad P_\infty$-algebra structures on $A$ and degree $-1$ self-maps $b$ of $\Schur {\operad C} A$ such that the following conditions are fulfilled:
  \begin{itemize}
    \item $b$ is a coderivation of the $\operad C$-coalgebra $\Schur {\operad C} A$.
    \item $b$ is a perturbation of the differential $d$ on $\Schur {\operad C} A$, i.e.\ $(d+b)^2 =0$.
    \item $b$ decreases weight, i.e.\ $b \big( \weight w {\Schur {\operad C} A} \big) \subseteq \weight {<w} {\Schur {\operad C} A}$.
  \end{itemize}
\end{proposition}

\begin{definition} \label{def:P bar}
  Let $A$ be a $\operad P_\infty$-algebra.
  The \emph{bar construction} of $A$ is the dg $\operad C$-coalgebra
  $$\Bar_{\operad P} A \defeq \bigl( \Schur {\operad C} A, d + b \bigr),$$
  where $d + b$ is as in \Cref{lemma:Pinfty_equiv_def}.
\end{definition}

\begin{remark}
   \Cref{lemma:Pinfty_equiv_def} subsumes \cref{lemma:Ainfty_equiv_def,lemma:Cinfty_equiv_def,lemma:Linfty_equiv_def}. Furthermore, we recover the constructions discussed in the previous section, up to suspension. Indeed, if $\operad P$ is equal to $\Ass$, $\Com$, or $\Lie$, and if $\operad C = \Koszuldual{\operad P}$, then the suspension of the bar construction $\Bar_{\operad P} A$ may be identified with
  $\Bar A$, $\Harr{*}(A)$, or $\CE{*}(A)$, respectively.
\end{remark}

\subsection{\texorpdfstring{$\infty$}{∞}-morphisms}
As in the previous section, we consider a dg operad $\operad P$ together with a resolution
$$\operad P_\infty \xlongto{\eq} \operad P$$
of the form $\operad P_\infty = \Cobar \operad C$ for a dg cooperad $\operad C$.
We will assume that $\operad C$ is connected in the sense that $\operad C(0) =0$ and $\operad C(1) = \QQ$.

A morphism of $\operad P_\infty$-algebras $A\to A'$ is a morphism of chain complexes that commutes with the $\operad P_\infty$-algebra structure maps. By relaxing this requirement up to coherent homotopy, one arrives at the notion of $\infty$-morphisms, or strongly homotopy maps (sometimes referred to as ``shmaps'' \cite{StasheffHalperin70}).
The following definition can be compared to \cite[\S10.2.2]{LodayVallette12}.

\begin{definition}
  Let $A$ and $A'$ be $\operad P_\infty$-algebras.
  A \emph{$\operad P_\infty$-morphism}, or \emph{$\infty$-mor\-phism}, from $A$ to $A'$, written
  $$f \colon A \longinftyto A',$$
  is a morphism $f\colon \Bar A \to \Bar A'$ of $\operad C$-coalgebras.
  
 The linear part of an $\infty$-morphism $f$ is the chain map
 $$f_1\colon A \longto A'$$
 given by extracting the weight one component of the map
 $$f\colon \bigoplus_{n\geq 1} \operad C(n) \otimes_{\Symm{n}} A^{\otimes n} \longto \bigoplus_{n\geq 1} \operad C(n) \otimes_{\Symm{n}} (A')^{\otimes n}.$$
 Since the coderivation $b$ on $BA$ decreases weight, one sees that $f_1$ is a chain map. We say that $f$ is an \emph{$\infty$-isomorphism} if $f_1$ is an isomorphism, and that $f$ is a \emph{$\infty$-quasi-isomorphism} if $f_1$ is a quasi-isomorphism.
\end{definition}
It should be noted that the notion of $\infty$-isomorphism is much more flexible than the usual notion of an isomorphism. See \cref{remark:infinity-isomorphism} below for an example that illustrates this point.

\begin{remark}
Note that the definition of $\infty$-morphisms makes sense, and is interesting, also in the case when $A$ and $A'$ are ordinary $\operad P$-algebras. For instance, the category $\mathbf{DASH}$ of associative dg algebras and strongly homotopy multiplicative maps has been studied in e.g.~\cite{GugenheimMunkholm74}.
\end{remark}

\begin{lemma} \label{lemma:rectification}
  Assume that $\operad C$ is equipped with a coaugmention.
  Then two $\operad P$-algebras $A$ and $A'$ are $\infty$-quasi-isomorphic if and only if they are quasi-isomorphic.
\end{lemma}

\begin{proof}
  The ``if'' direction is clear.
  For the ``only if'' direction it is enough to consider the case where there is a $\infty$-quasi-isomorphism $f \colon A \inftyto A'$, i.e.\ a quasi-isomorphism of $\operad C$-coalgebras $f \colon \Bar A \to \Bar A'$.
  The bar--cobar adjunction
  \[
  \begin{tikzcd}
    \Coalg[\operad C] \rar[bend left]{\Cobar}[swap, name = U]{} & \Alg[\operad P] \lar[bend left]{\Bar}[swap, name = D]{}
    \ar[from = U, to = D, phantom]{}{\vertdashv}
  \end{tikzcd}
  \]
  yields a zig-zag of maps of $\operad P$-algebras
  \[ A  \xlongfrom{\eq}  \Cobar \Bar A  \xlongto{\Cobar f}  \Cobar \Bar A'  \xlongto{\eq}  A' \]
  where the counit maps are quasi-isomorphism (cf.\ \cite[\S 11.3]{LodayVallette12}).
  To see that $\Cobar f$ is a quasi-isomorphism as well, we note that, since $\operad C$ is coaugmented, there is a commutative square
  \[
  \begin{tikzcd}
    \Cobar \Bar A  \rar{\Cobar f} & \Cobar \Bar A' \\
    A \rar{\eq}[swap]{f_1} \uar & A' \uar
  \end{tikzcd}
  \]
  where $f_1$ is the linear part of $f$, such that each of the two vertical maps is a section of the respective counit map, and hence a quasi-isomorphism itself.
\end{proof}

\subsection{Homotopy transfer theorem}
A principal advantage that $\operad P_\infty$-algebras have over $\operad P$-algebras is that they are homotopy invariant. The basic principles for homotopy invariant algebraic structures in topology were laid down by Boardman--Vogt \cite{BoardmanVogt73}. For an adaptation to the differential graded context, see \cite{Markl04}. One of the defining properties of homotopy invariant algebraic structures is that they can be transferred along homotopy equivalences. This is the main content of the homotopy transfer theorem.

There is an abundance of different approaches to, and variants of, the homotopy transfer theorem for $\operad P_\infty$-algebras and we will not attempt to give a complete account of all of these here.
Instead, we will give a short proof that uses homological perturbation theory, following \cite{Berglund14}.
A feature of this approach is that it works entirely with the compact description of $\operad P_\infty$-algebras as coderivation differentials. Moreover, it yields explicit formulas that are easier to work with than the tree formulas (as found in e.g.~\cite[\S10.3]{LodayVallette12}).

Homological perturbation theory is a tool that arose in the study of chain models for fibrations \cite{Brown65,Gugenheim72}.
The idea of using homological perturbation theory to transfer algebraic structures encoded by (co)derivation differentials was successfully realized for $\Ainfty$-algebras early on, see e.g.\ \cite{GLS91,HuebschmannKadeishvili91}.
The hurdles for realizing this idea for algebras over general operads, allowing the unified treatment presented here, were overcome in \cite{Berglund14}. The formulation and proof we give here are taken from \cite[Theorem 1.3]{Berglund14}.

\begin{theorem}[Homotopy transfer theorem]
  Let $\operad C$ be a dg cooperad which is connected in the sense that $\operad C(0) = 0$ and $\operad C(1) = \QQ$ and let $\operad P_\infty = \Cobar \operad C$. Let
  \[
  \begin{tikzcd}
    A \ar[loop, out = 150, in = 210, distance = 2em, "h"'] \rar[shift left]{f} & B \lar[shift left]{g}
  \end{tikzcd}
  \]
  be a contraction of chain complexes, i.e.\ $f$ and $g$ are degree $0$ chain maps, $h$ is a map of degree $1$, and the following equations hold: $fg = \id$, $gf = {\id} + \del(h)$, $fh = 0$, $hg = 0$, and $h^2 = 0$.
  For every $\operad P_\infty$-algebra structure $b$ on $A$, there is an induced $\operad P_\infty$-algebra structure $b'$ on $B$ and extensions of $f$ and $g$ to $\infty$-quasi-isomorphisms.
\end{theorem}

\begin{proof}
  The first step is to extend the given contraction to a contraction
  \[
  \begin{tikzcd}
    \Schur {\operad C} A \ar[loop, out = 160, in = 200, distance = 2em, "H"'] \rar[shift left]{F} & \Schur {\operad C} B . \lar[shift left]{G}
  \end{tikzcd}
  \]
  Here, $\Schur {\operad C} A$ and $\Schur {\operad C} B$ are equipped with the internal differentials (which are induced by the ones of $\operad C$, as well as $A$ and $B$, respectively).
  The maps $F$ and $G$ are defined in the obvious way, $F= \Schur {\operad C} f$ and $G = \Schur {\operad C} g$.
  It is less obvious how to define $H$, but the following works: Define a self-map of $A^{\tensor n}$ by
  $$h_n^\Sigma = \sum_{j=1}^n \sum_{\substack{\epsilon\in\{0,1\}^n \\ \epsilon_j = 0}} a_{n,|\epsilon|+1} \pi^{\epsilon_1} \tensor \ldots \tensor \pi^{\epsilon_{j-1}} \tensor h \tensor \pi^{\epsilon_{j+1}} \tensor \ldots \tensor \pi^{\epsilon_n},$$
  where $\pi = gf$, $|\epsilon| = \epsilon_1+\ldots+\epsilon_n$, and
  $$a_{n,k} = \frac{1}{\binom{n}{k}k}.$$
  Since $h_n^\Sigma$ is symmetric, it induces a self-map of $\operad C(n) \tensor[\Sigma_n] A^{\tensor n}$ and $H$ is defined by taking the sum of $h_n^\Sigma$ over all $n$. One can in principle check by hand that this yields a contraction, but for a more conceptual proof we refer to \cite[\S5]{Berglund14}.
  
  Next, the ``basic perturbation lemma'' (cf.~\cite{Brown65} or \cite[\S3]{Gugenheim72}), where the perturbation $b$ is the input, produces a new contraction,
  \[
  \begin{tikzcd}
    (\Schur {\operad C} A, d + b) \ar[loop, out = 170, in = 190, distance = 2em, "H'"'] \rar[shift left]{F'} & (\Schur {\operad C} B, d + b'), \lar[shift left]{G'}
  \end{tikzcd}
  \]
  given by the explicit formulas
  \begin{align*}
    b' &= F \Sigma G, \\
    F' &= F (1 + \Sigma H), \\
    H' &= H (1 + \Sigma H), \\
    G' &= (1 + H \Sigma) G,
  \end{align*}
  where
  $$\Sigma = \sum_{k = 0}^\infty b (H b)^k.$$
  Note that the above series is finite when evaluated on an element since $b$ decreases weight.
  
  It is a non-obvious fact that $b'$ is a coderivation and that $F'$ and $G'$ are morphisms of $\operad C$-coalgebras. Again, this can in principle be checked by hand, but a more conceptual explanation can be found in \cite{Berglund14}.
  
  Lastly we note that in weight $1$ the maps $F'$ and $G'$ are given by $f$ and $g$, respectively, so that the former actually extend the latter.
  This also implies that $F'$ and $G'$ are $\infty$-quasi-isomorphisms.
\end{proof}

\begin{remark}
The coefficients $a_{n,k}$ in the formula for $h_n^\Sigma$ are the entries in ``Leibniz's harmonic triangle'' (see \cref{fig:Leibniz_triangle}): they can be defined by the recursive formulas
  \begin{align*}
      a_{n,1} & = \frac{1}{n}, \\
      a_{n,k} & = a_{n-1,k-1} - a_{n,k-1},
  \end{align*}
for $1<k\leq n$. For example,
  \begin{align*}
      h_1^\Sigma &= h, \\
      h_2^\Sigma & = \frac{1}{2}\big(h\tensor 1 + 1\tensor h\big) + \frac{1}{2} \big(h\tensor \pi + \pi\tensor h\big), \\
      h_3^\Sigma & = \frac{1}{3}\big(h\tensor 1 \tensor 1 + 1\tensor h\tensor 1 + 1\tensor 1 \tensor h\big) \\
      & + \frac{1}{6}\big(h\tensor \pi \tensor 1 + h\tensor 1 \tensor \pi +
      \pi\tensor h\tensor 1 + 1\tensor h\tensor \pi + \pi \tensor 1 \tensor h + 1 \tensor \pi \tensor h\big) \\
      & + \frac{1}{3}\big(h\tensor \pi \tensor \pi + \pi\tensor h\tensor \pi + \pi\tensor \pi \tensor h\big).
  \end{align*}
\begin{figure}
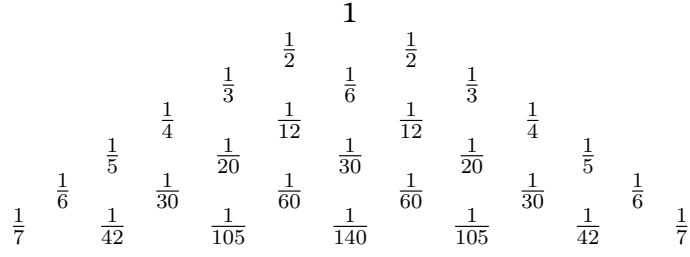

    \centering
\begin{tabular}{cccccccccccccc}
&&&&&&&1&&&&&&\\
&&&&&&$\frac{1}{2}$&&$\frac{1}{2}$&&&&&\\
&&&&&$\frac{1}{3}$&&$\frac{1}{6}$&&$\frac{1}{3}$&&&&\\
&&&&$\frac{1}{4}$&&$\frac{1}{12}$&&$\frac{1}{12}$&&$\frac{1}{4}$&&&\\
&&&$\frac{1}{5}$&&$\frac{1}{20}$&&$\frac{1}{30}$&&$\frac{1}{20}$&&$\frac{1}{5}$&&\\
&&$\frac{1}{6}$&&$\frac{1}{30}$&&$\frac{1}{60}$&&$\frac{1}{60}$&&$\frac{1}{30}$&&$\frac{1}{6}$&\\
&$\frac{1}{7}$&&$\frac{1}{42}$&&$\frac{1}{105}$&&$\frac{1}{140}$&&$\frac{1}{105}$&&$\frac{1}{42}$&&$\frac{1}{7}$
\end{tabular}
    \caption{Leibniz's harmonic triangle}
    \label{fig:Leibniz_triangle}
\end{figure}
\end{remark}

\subsection{Minimal \texorpdfstring{$\operad P_\infty$}{P∞}-algebra model} \label{sec:minimal}
A $\operad P_\infty$-algebra is called \emph{minimal} if its differential is trivial.
The property that justifies the terminology is that an $\infty$-quasi-isomorphism between minimal $\operad P_\infty$-algebras is an $\infty$-isomorphism. Also, as we will discuss in Section \ref{sec:nothing new} below, minimal $L_\infty$-algebras correspond to minimal Sullivan models and minimal $C_\infty$-algebras correspond to minimal Lie models.
The minimality theorem for $A_\infty$-algebras goes back Kadeishvili \cite{Kadeishvili82}.
The proof is an application of the homotopy transfer theorem.


\begin{theorem} \label{thm:minimal_Pinfty}
  Let $A$ and $A'$ be $\operad P$-algebras.
  \begin{itemize}
      \item There is a $\operad P_\infty$-algebra structure on $\Coho * (A)$ and $\infty$-quasi-isomorphisms from $A$ to $\Coho * (A)$ and vice versa.
      \item The $\operad P$-algebras $A$ and $A'$ are quasi-isomorphic if and only if $\Coho * (A)$ and $\Coho * (A')$ are $\infty$-isomorphic.
  \end{itemize}
\end{theorem}

\begin{proof}
  For the first part we use the homotopy transfer theorem: over a field, one can always find a contraction between a cochain complex $A$ and its cohomology $\Coho * (A)$ (this is well-known, see e.g.\ \cite[Exercise 1.4.4]{WeibelBook} or \cite[Lemma B.1]{BerglundMadsen20} and \cite[Remark 2.1]{Berglund14}).
  
  For the second part we first prove the ``only if'' direction.
  To this end it is enough to consider the case where there is a quasi-isomorphism $f \colon A \to A'$. Combined with the
   $\infty$-quasi-isomorphisms from the first part, we obtain $\infty$-quasi-isomorphisms
  \[ \Coho * (A)  \xlonginftyto{\eq} A  \xlongto{f}  A'  \xlonginftyto{\eq}  \Coho * (A') \]
  the composite of which is an $\infty$-quasi-isomorphism from $\Coho * (A)$ to $\Coho * (A')$. Since the source and target are minimal, this is necessarily an $\infty$-isomorphism.
  
  For the ``if'' direction we use that an $\infty$-isomorphism from $\Coho * (A)$ to $\Coho * (A')$, combined with the $\infty$-quasi-isomorphisms from the first part,
  \[ A  \xlonginftyto{\eq}  \Coho * (A)  \xlonginftyto{\iso}  \Coho * (A')  \xlonginftyto{\eq}  A' \]
  yield an $\infty$-quasi-isomorphism $f \colon A \inftyto A'$.
  By \cref{lemma:rectification}, this implies that $A$ and $A'$ are quasi-isomorphic as desired.
\end{proof}

\subsection{Sullivan's rational homotopy theory}
There are many good sources on Sullivan's approach to rational homotopy theory. In addition to the original \cite{Sullivan77}, comprehensive accounts can be found, for example, in any of \cite{BousfieldGugenheim76, FelixHalperinThomas01, Berglund12, GriffithsMorgan13}.
We will here give a condensed summary pointing out the key facts needed for proving \cref{thm:C-infinity model} and \cref{thm:L-infinity model}, referring to the literature for proofs and further details.

\begin{definition} \label{def:omega}
  We denote by $\PDRforms \bullet$ the simplicial cdga given by
  \[ \PDRforms{n}  \defeq  \frac{\SA ( t_0, \dots, t_n, dt_0, \dots, dt_n )}{( t_0 + \dots + t_n - 1, dt_0 + \dots + dt_n )} \]
  with $t_k$ in degree $0$ and $dt_k$ in cohomological degree $1$ (i.e.\ homological degree $-1$); here $\SA$ denotes the free graded commutative algebra on the given generators.
  The differential $d$ is, as the notation suggests, determined by $d(t_k) = dt_k$. For a morphism $\varphi\colon [m]\to [n]$ in the simplex category $\Delta$, the cdga morphism $\varphi^*\colon \PDRforms{n} \to \PDRforms{m}$ is determined by
  \[\varphi^*(t_i) = \sum_{j\in \varphi^{-1}(i)} t_j.\]
\end{definition}

We can think of $\PDRforms{n}$ as de Rham forms on an $n$-simplex that are ``polynomial''.
We can extend this construction to a general simplicial set $X$ by taking one copy of $\PDRforms{n}$ for each $n$-simplex and gluing them together according to the simplicial structure.
The resulting cdga, which we denote by $\PDR(X)$, is called \emph{polynomial de Rham forms} on $X$.

More abstractly, we obtain from $\PDRforms{\bullet}$, via the general nerve/realization construction, an adjunction
\begin{equation} \label{eq:pdr_real_adjunction}
\begin{tikzcd}
  \sSet \rar[bend left]{\PDR(\blank)}[swap, name = U]{} & \opcat{\CDGA} \lar[bend left]{\Realization{\blank}}[swap, name = D]{}
  \ar[from = U, to = D, phantom]{}{\vertdashv}
\end{tikzcd}
\end{equation}
where the functors are explicitly given by
\begin{align*}
  \PDR(X) &\defeq \colim{\Simplex {n} \to X} \PDRforms {n} \\
  \Realization{A} &\defeq \Hom[\CDGA] (A, \PDRforms {\bullet})
\end{align*}
We call $\Realization{A}$ the \emph{realization} of the cdga $A$.

Since $\PDR (X)$ is a simplicial analogue of the de Rham complex, one would expect it to compute the (rational) cohomology of $X$.
This is the content of the following theorem; see \cite[Theorem 10.15]{FelixHalperinThomas01} or
\cite[Theorem 2.2]{BousfieldGugenheim76} for a proof.

\begin{theorem}[Polynomial de Rham theorem] \label{thm:polyDeRham}
  Let $X$ be a simplicial set.
  Integration of polynomial differential forms defines a natural quasi-isomorphism of cochain complexes
  \[ I\colon \PDR(X) \xlongto{\eq} \Cochains{*}(X; \QQ) \]
  from polynomial de Rham forms to singular cochains. Explicitly,
  \[I(\omega)(\sigma) = \int_{\Delta^n} \omega_\sigma,\]
  for $\omega\in \Omega^*(X)$ and $\sigma\colon \Delta^n \to X$.
\end{theorem}

\begin{remark} \label{remark:pdr}
  The quasi-isomorphism $I$ in the preceding theorem is not a morphism of dg algebras, but \cite[Proposition 3.3]{BousfieldGugenheim76} shows that it extends to an $A_\infty$-quasi-isomorphism. In particular, this implies that the induced map in cohomology,
  \[\Coho{*}(I)\colon \Coho{*}(\Omega^*(X)) \to \Coho{*}(X;\QQ),\]
  is an isomorphisms of algebras. By \Cref{thm:minimal_Pinfty}, this also implies that $\PDR(X)$ is quasi-isomorphic to $\Cochains{*}(X;\QQ)$ as an associative dg-algebra. For an alternative proof of this statement, see \cite[Corollary 10.10]{FelixHalperinThomas01}.
\end{remark}

\begin{remark} \label{remark:dupont}
In his proof of the de Rham theorem, Dupont \cite[\S2]{Dupont78} constructs a contraction
\begin{equation} \label{eq:dupont}
  \begin{tikzcd}
     \PDR(X) \ar[loop, out = 160, in = 200, distance = 2em, "s"'] \rar[shift left]{I} & \Cochains{*}(X;\QQ) \lar[shift left]{E}
  \end{tikzcd}
\end{equation}
that is \emph{natural} in $X$. As observed by Cheng--Getzler \cite{ChengGetzler08}, an application of the homotopy transfer theorem for $\Cinfty$-algebras then shows that the cochain complex $\Cochains{*}(X;\QQ)$ carries a natural $\Cinfty$-algebra structure, which is naturally $\Cinfty$-quasi-isomorphic to $\PDR(X)$. The existence of such a $\Cinfty$-algebra structure on $\Cochains{*}(X;\QQ)$ has also been noticed by Sullivan using different methods, see the appendix of \cite{TradlerZeinalian07}.
\end{remark}

Next, we review the definition of Sullivan algebras. These play the role of CW-complexes in the category of cdgas: they are cofibrant and every connected cdga can be replaced by a Sullivan algebra up to quasi-isomorphism.

\begin{definition}
  A \emph{Sullivan algebra} is a cdga of the form $(\SA V, d)$, where $V$ is a graded vector space concentrated in positive cohomological degrees and the differential $d$ satisfies the following \emph{nilpotence condition}: there exists an exhaustive filtration
  \[ 0 = F_{-1} V \subseteq F_0 V \subseteq F_1 V \subseteq \dots \subseteq V \]
  such that $d(F_p V) \subseteq \SA (F_{p-1} V)$ for all $p \ge 0$.
  
  A Sullivan algebra $(\SA V, d)$ is called \emph{minimal} if $d(V) \subseteq \SA[\ge 2] V$.
  It is called \emph{of finite type} if $V$ (or, equivalently, $\SA V$) is of finite type.
\end{definition}

Every quasi-isomorphism between minimal Sullivan algebras is an isomorphism (see \cite[Theorem 14.11]{FelixHalperinThomas01}). This justifies the terminology ``minimal''. The reader may compare with minimal Kan complexes in simplicial homotopy theory (see e.g.\ \cite[Chapter I.10]{GoerssJardine09}) or minimal resolutions in homological algebra (see e.g.\ \cite[Proposition 1.1.2]{avramov98}).

\begin{remark}
  Sullivan algebras are cofibrant.
  In particular this implies that, for any quasi-isomorphism $q \colon A \to B$ of cdgas and any map $f \colon (\SA V, d) \to B$, there exists a lift $\tilde f$ such that the diagram
  \[
  \begin{tikzcd}
    & A \dar{\eq}[swap]{q} \\
    (\SA V, d) \rar{f} \urar[dashed]{\tilde f} & B
  \end{tikzcd}
  \]
  commutes up to homotopy, and this lift is unique up to homotopy (see e.g.\ \cite[Proposition 12.9]{FelixHalperinThomas01}).
\end{remark}

The following theorem states that (connected) spaces can always be modeled by minimal Sullivan algebras.
A proof can be found, for example, in \cite[Corollary, p.191]{FelixHalperinThomas01} or \cite[Proposition 7.7]{BousfieldGugenheim76}.

\begin{theorem}[Existence of minimal models] \label{thm:minimal_Sullivan_model}
  Let $X$ be a connected space.
  Then there exists a unique, up to non-canonical isomorphism, minimal Sullivan algebra $\mathcal M_X = (\SA V, d)$ that is quasi-isomorphic to $\PDR(X)$.
\end{theorem}

\begin{definition}
  The Sullivan algebra $\mathcal M_X$ of \Cref{thm:minimal_Sullivan_model} is called the \emph{minimal Sullivan model} of $X$.
\end{definition}

\begin{remark}
If $X$ is nilpotent and of finite $\QQ$-type, then the minimal model $\mathcal M_X$ is of finite type, see \cite[Theorem 10.1]{BousfieldGugenheim76}.
\end{remark}

The next theorem states that the cohomology of the realization of a finite type Sullivan algebra agrees with the cohomology of the algebra. See \cite[Theorem 8.1]{Sullivan77}, \cite[Theorem 17.10]{FelixHalperinThomas01}, or \cite[Theorem 10.1]{BousfieldGugenheim76}.

\begin{theorem}[Cohomology of the spatial realization]
  Let $(\SA V, d)$ be a Sullivan algebra of finite type.
  Then the unit $(\SA V, d) \to \PDR \Realization{(\SA V, d)}$ of the adjunction \eqref{eq:pdr_real_adjunction} is a quasi-isomorphism.
\end{theorem}

\begin{exercise} \label{ex:rat_eq_pdr}
  Using the facts stated above, show that two nilpotent spaces $X$ and $Y$ of finite $\QQ$-type are rationally equivalent if and only if $\PDR (X)$ and $\PDR (Y)$ are quasi-isomorphic as cdgas.
\end{exercise}

Granted this, the proof of \Cref{thm:C-infinity model} is straightforward.

\begin{proof}[Proof of \Cref{thm:C-infinity model}]
  By \cref{ex:rat_eq_pdr} two simply connected spaces $X$ and $Y$ are rationally equivalent if and only if $\PDR (X)$ and $\PDR (Y)$ are quasi-isomorphic as cdgas.
  By \cref{thm:minimal_Pinfty,thm:polyDeRham} the latter condition is equivalent to $\Coho{*} (X; \QQ)$ and $\Coho{*} (Y; \QQ)$ being $\Cinfty$-isomorphic.
\end{proof}

\subsection{Quillen's rational homotopy theory}

Quillen \cite{Quillen69} defined a functor
$$\lambda\colon \Top_1 \longto \DGL_{\geq 1}$$
from the category of simply connected pointed topological spaces to the category of positively graded dg Lie algebras over $\QQ$, and showed that it induces an equivalence of categories after formally inverting the rational equivalences in $\Top_1$ and the quasi-isomorphisms in $\DGL_{\geq 1}$. In particular, $X$ and $Y$ are rationally equivalent if and only if $\lambda X$ and $\lambda Y$ are quasi-isomorphic.
Another feature of Quillen's functor is the existence of a natural isomorphism
$$\Ho{*}{} (\lambda X) \cong \htpygrp{*+1}(X)\tensor \QQ.$$
Granted these facts, the proof of \Cref{thm:L-infinity model} is a simple application of the homotopy transfer theorem for $L_\infty$-algebras.

\begin{proof}[Proof of \Cref{thm:L-infinity model}]
  Applied to $\lambda X$, the homotopy transfer theorem yields an $L_\infty$-algebra structure on $\Ho{*}{}(\lambda X) \cong \htpygrp{*+1}(X)\tensor \QQ$.
  Next, by Quillen's theory, the spaces $X$ and $Y$ are rationally equivalent if and only if $\lambda X$ and $\lambda Y$ are quasi-isomorphic, which happens if and only if $\htpygrp{*+1}(X)\tensor \QQ$ and $\htpygrp{*+1}(Y)\tensor \QQ$ are $\Linfty$-isomorphic by \cref{thm:minimal_Pinfty}.
\end{proof}

Counterparts of Sullivan's minimal models for dg Lie algebras were developed by Baues--Lemaire \cite{BauesLemaire77} and Neisendorfer \cite{Neisendorfer78}. We will stick to the simply connected case here for simplicity, but we should mention that generalizations to non-nilpotent spaces have been developed recently by Buijs--F\'elix--Murillo--Tanr\'e \cite{BFMT20}.

A dg Lie algebra is called \emph{minimal} if it is of the form $(\freeLie V, \delta)$, where $\freeLie V$ is the free Lie algebra on a positively graded vector space $V$ and $\delta$ is decomposable in the sense that
$$\delta(V)\subseteq \freeLie[\geq 2] V.$$
Every positively graded dg Lie algebra admits a unique, up to isomorphism, minimal model, see e.g.\ \cite[Theorem~2.3]{BauesLemaire77}, \cite[Proposition~5.6(c)]{Neisendorfer78} or \cite[Theorem~22.13]{FelixHalperinThomas01}. Also see \cite[Theorem 3.19]{BFMT20} for a generalization to non-negatively graded complete dg Lie algebras. Applied to Quillen's $\lambda X$, this yields the following result.

\begin{theorem} \label{thm:quillen model}
For each simply connected space $X$, there is a unique, up to non-canonical isomorphism, minimal dg Lie algebra $\LL_X$ such that
$$\LL_X \simeq \lambda X.$$
\end{theorem}

\begin{definition}
We will call the minimal dg Lie algebra $\LL_X$ in the above theorem the \emph{minimal Quillen model} of $X$.
\end{definition}

\subsection{Nothing new under the sun: \texorpdfstring{$\infty$}{∞}-algebras vs.\ minimal models} \label{sec:nothing new}

As hinted in the introduction, the homotopy transfer and minimality theorems were known long before $\infty$-algebras became part of mainstream mathematics, but in a different language.
Finite type nilpotent $\Linfty$-algebras are essentially the same thing as Sullivan algebras of finite type.
More formally we have the following statement, a proof of which can be found e.g.\ in \cite[Theorem 2.3]{Berglund15}.
Recall, from \cref{sec:minimal}, that an $\Linfty$-algebra is called minimal if its differential is trivial. We say that an $\Linfty$-algebra $L$ is \emph{nilpotent} if its lower central series terminates degreewise, see \cref{def:nilpotent}.

\begin{proposition} \label{prop:L-infinity dictionary}
The functor $\coCE * (\blank)$ restricts to an equivalence of categories
$$
\left\{ \parbox{4cm}{\center Nilpotent $L_\infty$-algebras $L$ of finite type \endcenter} \right\} \longrightarrow \left\{ \parbox{4cm}{\center Sullivan algebras $(\Lambda V,d)$ of finite type \endcenter} \right\},
$$
where the morphisms are $\infty$-morphisms of $\Linfty$-algebras and morphisms of cdgas, respectively. Moreover, a nilpotent $\Linfty$-algebra $L$ is minimal if and only if the Sullivan algebra $\coCE * (L)$ is minimal.
\end{proposition}

\begin{remark}
  Under this dictionary, the minimality theorem for $L_\infty$-algebras (\cref{thm:minimal_Pinfty}) is nothing but the classical result that a Sullivan algebra has a minimal model, cf.~\cite[Theorem 2.2]{Sullivan77} or \cite[Theorem 14.9]{FelixHalperinThomas01}.
\end{remark}

A key feature of minimal Sullivan models is that the rational homotopy groups can be read off easily.
This is the content of the following statement, a proof of which can be found in, for example, \cite[Theorem 10.1(i)]{Sullivan77} or \cite[Theorem 15.11]{FelixHalperinThomas01}.
For a proof using nerves of $L_\infty$-algebras, see \cref{remark:homotopy groups}.

\begin{proposition} \label{prop:homotopy groups from sullivan model}
  If $X$ is a simply connected space of finite $\QQ$-type with minimal Sullivan model
  $$\mathcal M_X = (\SA V, d) \eq \PDR(X),$$
  then $V^k$ is dual to $\htpygrp{k}(X) \tensor \QQ$ for every $k$.
\end{proposition}

The preceding two propositions together imply that the differential $d$ of the minimal Sullivan model $\mathcal M_X$ corresponds to a minimal nilpotent $\Linfty$-algebra structure on $L_X \defeq \htpygrp{* + 1}(X) \tensor \QQ$ such that $\coCE * (L_X) \iso \mathcal M_X$.
We can use this to give an alternative proof of \cref{thm:L-infinity model}.

\begin{proof}[Second proof of \Cref{thm:L-infinity model}]
  By \cref{thm:minimal_Sullivan_model,ex:rat_eq_pdr}, two simply connected spaces $X$ and $Y$ of finite $\QQ$-type are rationally equivalent if and only if their minimal Sullivan models $\mathcal M_X$ and $\mathcal M_Y$ are isomorphic.
  By the preceding discussion this is equivalent to $\htpygrp{* + 1}(X) \tensor \QQ$ and $\htpygrp{* + 1}(Y) \tensor \QQ$ being $\Linfty$-isomorphic.
\end{proof}

\begin{remark}
It is not a priori clear that the
$\Linfty$-structure on $\htpygrp{*+1}(X)\tensor \QQ$ obtained from the minimal Sullivan model via the equivalence of categories in Proposition \ref{prop:L-infinity dictionary} is $\Linfty$-isomorphic to the one obtained from Quillen's dg Lie algebra $\lambda X$ via the homotopy transfer theorem. The statement that they are is equivalent to a conjecture formulated by Baues--Lemaire \cite[Conjecture 3.5]{BauesLemaire77}, proved by Majewski \cite[Theorem 4.90]{Majewski00}.
See also \cite{Berglund24} for a short proof.
\end{remark}

The following analog of \cref{prop:L-infinity dictionary} is a straightforward consequence of \cref{lemma:Cinfty_equiv_def}. We say that a non-unital $\Cinfty$-algebra is \emph{simply connected} if it is concentrated in cohomological degrees $\geq 2$.

\begin{proposition} \label{prop:C-infinity dictionary}
The functor $\coHarr{*}(\blank)$ restricts to an equivalence of categories
$$
\left\{ \parbox{4cm}{\center Finite type, simply connected $\Cinfty$-algebras $A$ \endcenter} \right\} \longrightarrow \left\{ \parbox{4cm}{\center Finite type dg Lie algebras $(\LL V,\delta)$ with $V=V_{\geq 1}$ \endcenter} \right\},
$$
where the morphisms are $\infty$-morphisms of $\Cinfty$-algebras and morphisms of dg Lie algebras, respectively. Moreover, a $\Cinfty$-algebra $A$ is minimal if and only if the dg Lie algebra $\coHarr * (A)$ is minimal.
\end{proposition}

The following is dual to \cref{prop:homotopy groups from sullivan model}. For a proof, see \cite[Proposition 24.4]{FelixHalperinThomas01}.

\begin{proposition} \label{prop:quillen model indec}
If $X$ is a simply connected space with minimal Quillen model
$$\LL_X = (\freeLie V,\delta) \simeq \lambda X,$$
then $V_k \cong \rHo{k+1}(X;\QQ)$ for all $k$.
\end{proposition}

In particular, \cref{prop:quillen model indec} together with \cref{prop:C-infinity dictionary} imply that the differential in the minimal Quillen model $\LL_X$ of a simply connected space $X$ of finite $\QQ$-type corresponds to a minimal $\Cinfty$-algebra structure on the cohomology $\rCoho{*}(X;\QQ)$ and, moreover, $\LL_X$ may be identified with the Harrison cochains $\coHarr{*}(\rCoho{*}(X;\QQ))$ on this $\Cinfty$-algebra. This leads to an alternative proof of \cref{thm:C-infinity model}.

\begin{proof}[Second proof of \Cref{thm:C-infinity model}]
  By \cref{thm:quillen model}, two simply connected spaces $X$ and $Y$ of finite $\QQ$-type are rationally equivalent if and only if their minimal Quillen models $\LL_X$ and $\LL_Y$ are isomorphic. Since we may identify $\LL_X$ with the construction $\coHarr{*}(\rCoho{*}(X;\QQ))$, this is equivalent to $\Coho{*}(X;\QQ)$ and $\Coho{*}(Y;\QQ)$ being $\Cinfty$-isomorphic.
\end{proof}

We can summarize the situation in a diagram
$$
\xymatrix{\textrm{Sullivan model $\big(\Lambda V,d\big)$}  \ar[r]^-{\text{HTT}} & \textrm{$\Cinfty$-algebra $\rCoho{*}(X;\QQ)$} \ar[d]^-{\coHarr{*}(-)} \\
\textrm{$\Linfty$-algebra $\htpygrp{*+1}(X) \tensor \QQ$} \ar[u]^-{\coCE{*}(-)} & \ar[l]_-{\text{HTT}} \textrm{Quillen model $\big(\freeLie(W),\delta\big)$}} 
$$
of different models of a simply connected space $X$.



\section{Lecture 2: Koszul algebras, formality, and coformality}

The notion of formality has long been a central one in rational homotopy theory. A landmark paper featuring the notion is \cite{DGMS75}.
The notion of coformality, introduced in \cite{NeisendorferMiller78}, has not been studied to the same extent, perhaps because it is not as clearly linked to geometric structures on the space, but examples of coformal spaces abound. Koszul algebras were introduced by Priddy \cite{Priddy70} as a tool for computing cohomology algebras $\Ext_A^*(k,k)$ of augmented associative algebras $A \to k$.

In this lecture, we will review the interplay between the notions of formality, coformality and Koszul algebras that was elucidated in \cite{Berglund14bis}. We will also adapt the notion of Koszul $\Ainfty$-algebras that was introduced in \cite{BerglundBorjeson20} to $\Cinfty$- and $\Linfty$-algebras and establish analogs of the results obtained there.

\subsection{Formality and coformality}

We begin by defining the titular notions.

\begin{definition} \label{def:(co)formality}
  Let $X$ be a simply connected space of finite $\QQ$-type.
  \begin{itemize}
    \item $X$ is called \emph{formal} if $\PDR(X)$ is quasi-isomorphic to $\Coho * (X; \QQ)$ as a cdga.
    \item $X$ is called \emph{coformal} if $\Quillen X$ is quasi-isomorphic to $\htpygrp{* + 1}(X) \tensor \QQ$ as a dg Lie algebra.
  \end{itemize}
\end{definition}

\begin{example} \label{ex:cpt_Kahler}
~
\begin{enumerate}
\item A compact Kähler manifold is formal by the work of Deligne--Griffiths--Morgan--Sullivan \cite{DGMS75}. In particular, smooth complex projective varieties are examples of formal spaces.

\item A theorem of Miller \cite{Miller79} says that for all $k>1$, every $(k-1)$-connected Poincar\'e duality space of dimension at most $4k-2$ is formal. In particular, all simply connected closed manifolds of dimension at most $6$ are formal.

\item A simply connected space $X$ is called \emph{rationally elliptic} if $\Coho{*}(X;\QQ)$ and $\htpygrp{*}(X)\otimes \QQ$ are finite dimensional. Elliptic spaces with positive Euler characteristic are formal. In fact, the cohomology ring of such a space has the form $\QQ[x_1,\ldots,x_n]/(f_1,\ldots,f_n)$, where $x_1,\ldots, x_n$ are even dimensional classes and $f_1,\ldots,f_n$ is a regular sequence \cite[Proposition 32.16(ii)]{FelixHalperinThomas01}. A space whose cohomology ring has this form is automatically formal, see \cite[\S 16]{BousfieldGugenheim76}. Examples of elliptic spaces with positive Euler characteristic include complex projective spaces and Grassmannians, or more generally homogeneous spaces $G/H$ for $H$ a closed subgroup of maximal rank in a connected compact Lie group $G$.

\item The arguably simplest example of a non-formal Sullivan algebra is
$$\big(\Lambda(x,y,z),d\big),$$
where $x$ and $y$ are cocycles of odd degree $k$ and $z$ is a generator of degree $2k-1$ such that $dz = xy$. Formality is obstructed by the existence of non-trivial Massey products, e.g.\ $\langle [x],[x],[y] \rangle \ne 0$. This Sullivan algebra represents an elliptic Poincar\'e duality space of dimension $4k-1$ with Euler characteristic zero, so it shows that the bound in Miller's theorem is sharp, and it shows that the non-vanishing of the Euler characteristic is necessary in the above example. This space is however coformal. See \cref{ex:nonformal} for a further discussion.

\item K\"ahler manifolds are examples of symplectic manifolds, but symplectic manifolds need not be formal in general, see \cite{BabenkoTaimanov00,FernandezMunoz08} and \cite[p.163]{FelixHalperinThomas01}.
\end{enumerate}
\end{example}

Whether a space is (co)formal is reflected by certain algebraic properties of its various rational models. This is made precise by the following two propositions.

We say that the differential in a minimal Quillen model $(\LL V,d)$ is quadratic if it satisfies $d(V) \subseteq \LL^{2}V$. Similarly, the differential of a minimal Sullivan model $(\Lambda V,d)$ is said to be quadratic if $d(V) \subseteq  \Lambda^{2} V$.

\begin{proposition} \label{prop:formality conditions}
  Let $X$ be a simply connected space of finite $\QQ$-type.
  Then the following are equivalent
  \begin{enumerate}
      \item The space $X$ is formal.
      \item The $\Cinfty$-algebra $\Coho{*}(X;\QQ)$ is $\Cinfty$-isomorphic to one with $m_n=0$ for $n\ne 2$.
      \item The space $X$ admits a minimal Quillen model with quadratic differential.
  \end{enumerate}
\end{proposition}

\begin{proposition} \label{prop:coformality conditions}
  Let $X$ be a simply connected space of finite $\QQ$-type.
  Then the following are equivalent
  \begin{enumerate}
      \item The space $X$ is coformal.
      \item The $\Linfty$-algebra $\htpygrp{*+1}(X)\tensor \QQ$ is $\Linfty$-isomorphic to one with $\ell_n=0$ for $n\ne 2$.
      \item The space $X$ admits a minimal Sullivan model with quadratic differential.
  \end{enumerate}
\end{proposition}

\begin{proof}
  We prove \cref{prop:formality conditions}. 
  The proof of \cref{prop:coformality conditions} is entirely analogous.
  An application of the second part of \cref{thm:minimal_Pinfty} to the cdgas $A = \PDR(X)$ and $A'= \Coho{*}(X;\QQ)$ shows that $X$ is formal if and only if $\Coho{*}(X;\QQ)$, with the transferred $\Cinfty$-algebra structure, is $\Cinfty$-isomorphic to $\Coho{*}(X;\QQ)$ with the trivial $\Cinfty$-algebra structure.
  For the equivalence with the last condition, one simply notes that under the equivalence of categories in \cref{prop:C-infinity dictionary}, a $\Cinfty$-algebra $A$ has $m_n = 0$ for $n\ne 2$ precisely when the differential of $\coHarr{*}(A)$ is quadratic. 
\end{proof}

\begin{remark} \label{remark:infinity-isomorphism}
  One should be careful when talking about ``the'' $\Cinfty$-structure on $\Coho{*}(X;\QQ)$, because it is only defined up to $\Cinfty$-isomorphism, a notion which is much more flexible than the usual notion of isomorphism.
  
  For example, consider the space
  $$X = S^2 \vee \big(S^2 \times S^3\big).$$
  It has minimal Quillen model
  $$\big( \freeLie(\alpha,\beta,\gamma,\xi), \delta \big),$$
  where $|\alpha| = |\beta| = 1$, $|\gamma| = 2$, and $|\xi| = 4$. The differential $\delta$ is trivial on $\alpha$, $\beta$, and $\gamma$, but
  $$\delta(\xi) = [\alpha,\gamma].$$
  
  Another minimal Quillen model is given by
  $$\big( \freeLie(\alpha,\beta,\gamma,\xi), \delta' \big),$$
  where $\delta'$ agrees with $\delta$ except
  $$\delta'(\xi) = [\alpha,\gamma] + [\alpha,[\alpha,\beta]].$$
  An isomorphism $\varphi\colon (\LL,\delta) \to (\LL,\delta')$ is determined by
  $$\varphi(\gamma) = \gamma + [\alpha,\beta]$$
  and by letting $\varphi$ act by the identity on the other generators.
  
  Under the equivalence of \cref{prop:C-infinity dictionary}, the differentials $\delta$ and $\delta'$ correspond to two different $C_\infty$-structures $m$ and $m'$ on $\Coho{*}(X;\QQ)$ and $\varphi$ corresponds to a $\Cinfty$-isomorphism between them. In the first $m_n = 0$ for all $n\ne 2$, but the presence of the cubic term in $\delta'(\xi)$ means that $m_3'\ne 0$.
\end{remark}

\begin{remark}
As discussed in \cref{remark:pdr}, the dg-algebra $\PDR(X)$ is quasi-isomorphic to $\Cochains{*}(X;\QQ)$. In particular, if $X$ is formal, then $\Cochains{*}(X;\QQ)$ is formal as an associative dg algebra. Since $\Cochains{*}(X;\QQ)$ is in general non-commutative, it is not obvious that the converse should hold, but this has been proved by Saleh \cite{Saleh17}. Similarly, Saleh shows that $X$ is coformal if and only if the associative dg algebra $\Chains{*}(\Loops X;\QQ)$ is formal.
A generalization was obtained by Campos--Petersen--Robert-Nicoud--Wierstra \cite{CPRW19}, who proved that two cdgas are quasi-isomorphic if and only if their underlying associative dg algebras are quasi-isomorphic.

\end{remark}

\subsection{Koszul algebras and simultaneous formality and coformality}
In this section, we will examine how the notions of formality, coformality, and Koszul algebras interact, following \cite{Berglund14bis}.

In what follows, $A$ will be a graded commutative non-unital algebra (e.g.\ $\rCoho{*} (X; \QQ)$ for some space $X$) and $L$ a graded Lie algebra (e.g.\ $\htpygrp{* + 1}(X)\tensor \QQ$).

\begin{definition}
  A \emph{weight grading} on a graded commutative non-unital algebra $A$ is a decomposition
  \[ A = \Dirsum_{w = 1}^\infty \weight w A \]
  such that $\weight p A \cdot \weight q A \subseteq \weight {p+q} A$.
  We consider it to be a cohomological grading (note that it is, accordingly, written as a superscript).
\end{definition}

A weight grading on $A$ induces a weight grading on
\[ \Harr{*} (A)  =  \bigl( \cofreeLie(\shift A), b \bigr) \]
by letting $\shift$ have weight $-1$ (here $b$ is as in \Cref{lemma:Cinfty_equiv_def}; note that the $d$ occurring there is trivial since $A$ has no differential).
We obtain a cochain complex
\begin{equation} \label{eq:Harr_weight_cmplx}
  \weight 0 {\Harr * (A)}  \xlongto{b}  \weight 1 {\Harr * (A)}  \xlongto{b}  \dots
\end{equation}
since $b$ increases weight by $1$.

\begin{definition}
  A graded commutative non-unital algebra $A$ is \emph{Koszul} if it admits a weight grading such that the cochain complex \eqref{eq:Harr_weight_cmplx} is exact, i.e.\ if $\Coho * (\Harr * (A)) = \ker b \intersect \weight 0 {\Harr * (A)}$.
\end{definition}

Next, we will see that Koszul algebras admit presentations of the following special form.

\begin{definition}
  A \emph{quadratic presentation} of a graded commutative non-unital algebra $A$ is a graded vector space $V$ together with a surjective map of graded algebras
  $$\pi \colon \SA V \longto A$$
  such that its kernel $I = \ker \pi$ is generated as an ideal by the kernel $R$ of the restriction $\SA[2] V \subseteq \SA V \to A$.
  
  A graded commutative non-unital algebra $A$ is \emph{quadratic} if it has some quadratic presentation.
\end{definition}

\begin{remark} \label{rem:quad_alg_unique}
  The quadratic presentation of a quadratic algebra $A$ is unique up to isomorphism.
  More precisely we have that, given two quadratic presentations $p \colon \SA V \to A$ and $q \colon \SA W \to A$ of $A$, there exists an isomorphism $f \colon V \to W$ of graded vector spaces and an automorphism $g$ of the algebra $A$ such that $q \after \SA f = g \after p$.
  
  To prove this, choose a lift $\tilde p \colon \SA V \to \SA W$ of $p$ along $q$ and check that its homogeneous part $f = \tilde p_0 \colon V \to W$ is an isomorphism of graded vector spaces such that $\SA f$ descends to an algebra isomorphism $g$ of $A$.
\end{remark}

\begin{remark}
  A quadratic presentation of $A$ induces a weight grading on $A$ by setting $\weight w A$ to be $\pi(\SA[w] V)$.
\end{remark}

The following proposition is a special case of \cite[Theorem 2.11]{Berglund14bis}.

\begin{proposition} \label{prop:alg_Koszul}
  Let $A$ be a graded commutative non-unital algebra.
  If $A$ is Koszul, then $A$ is quadratic.
  
  More explicitly, let $A = \Dirsum_{w = 1}^\infty \weight w A$ be a weight grading such that the cochain complex \eqref{eq:Harr_weight_cmplx} is exact.
  Then, writing $V \defeq \weight 1 A$, the multiplication map $\SA V \to A$ is a quadratic presentation of $A$, i.e.\ we have two short exact sequences
  \begin{gather}
    0 \longto I \longto \SA V \longto A \longto 0 \nonumber \\
    0 \longto R \longto \SA[2] V \longto \weight 2 A \longto 0 \label{eq:alg_Koszul_ses}
  \end{gather}
  such that $I$ is the ideal generated by $R = \weight 2 I$.
\end{proposition}

We will now discuss the Koszul property for Lie algebras.

\begin{definition}
  A \emph{weight grading} on a graded Lie algebra $L$ is a decomposition
  \[ L = \Dirsum_{w = 1}^\infty \weight w L \]
  such that $[\weight p A, \weight q A] \subseteq \weight {p+q} A$.
  We consider it to be a cohomological grading (note that it is, accordingly, written as a superscript).
\end{definition}

A weight grading on $L$ induces a weight grading on
\[ \CE * (L) = \bigl( \coSA(\shift L), b \bigr) \]
by letting $\shift$ have weight $-1$ (here $b$ is as in \Cref{lemma:Linfty_equiv_def}; note that the $d$ occurring there is trivial since $L$ has no differential).
We obtain a cochain complex
\begin{equation} \label{eq:CE_weight_cmplx}
  \weight 0 {\CE * (L)}  \xlongto{b}  \weight 1 {\CE * (L)}  \xlongto{b}  \dots
\end{equation}
since $b$ increases weight by $1$.

\begin{definition}
  A graded Lie algebra $L$ is \emph{Koszul} if it has some weight grading such that the cochain complex \eqref{eq:CE_weight_cmplx} is exact, i.e.\ if $\Coho * (\CE * (L)) = \ker b \intersect \weight 0 {\CE * (L)}$.
\end{definition}

Next, we will see that Koszul graded Lie algebras admit presentations of the following special form.

\begin{definition}
  A \emph{quadratic presentation} of a graded Lie algebra $L$ is a graded vector space $W$ together with a surjective map of graded Lie algebras $$\pi \colon \freeLie W \longto L$$
  such the kernel $J = \ker \pi$ is generated by the kernel $S$ of the restriction $\freeLie[2] W \subseteq \freeLie W \to L$.
  
  A graded Lie algebra $L$ is \emph{quadratic} if it has some quadratic presentation.
\end{definition}

\begin{remark} \label{rem:quad_lie_unqiue}
  As in \Cref{rem:quad_alg_unique}, the quadratic presentation of a quadratic Lie algebra $L$ is unique up to isomorphism.
\end{remark}

\begin{remark}
  A quadratic presentation of $L$ induces a weight grading on $L$ by setting $\weight w L$ to be $\pi(\freeLie[w] W)$.
\end{remark}

The following proposition is a special case of \cite[Theorem 2.11]{Berglund14bis}.

\begin{proposition} \label{prop:Lie_Koszul}
  Let $L$ be a graded Lie algebra.
  If $L$ is Koszul, then $L$ is quadratic.
  
  More explicitly, let $L = \Dirsum_{w = 1}^\infty \weight w L$ be a weight grading such that the cochain complex \eqref{eq:CE_weight_cmplx} is exact.
  Then, writing $W \defeq \weight 1 L$, the induced Lie algebra map $\freeLie W \to L$ is a quadratic presentation of $L$, i.e.\ we have two short exact sequences
  \begin{gather*}
  0 \longto J \longto \freeLie W \longto L \longto 0 \nonumber \\
  0 \longto S \longto \freeLie[2] W \longto \weight 2 L \longto 0 
  \end{gather*}
  such that $J$ is the Lie ideal generated by $S = \weight 2 J$.
\end{proposition}

We now discuss the Koszul dual of a quadratic commutative algebra or Lie algebra.

\begin{definition}
  We say that a graded commutative non-unital algebra $A$ and a graded Lie algebra $L$ are \emph{Koszul dual} if there exist some quadratic presentations $\SA V \to A$ and $\freeLie W \to L$ such that there is a nondegenerate pairing of degree $-1$
  \[ \pairing \blank \blank  \colon  V \tensor W  \longto  \QQ \]
  such that under the induced degree $-2$ pairing $\SA[2] V \tensor \freeLie[2] W  \to  \QQ$
  given by
  \begin{equation} \label{eq:weight_2_pairing}
  \begin{multlined}
    \pairing{x \wedge y}{[\alpha, \beta]}  \defeq \\ (-1)^{\deg y \deg \alpha + \deg x + \deg \alpha} \pairing x \alpha \pairing y \beta - (-1)^{\deg \alpha \deg \beta + \deg y \deg \beta + \deg x + \deg \beta} \pairing x \beta \pairing y \alpha
  \end{multlined}
  \end{equation}
  we have $R = \orth S$ with $R$ and $S$ as in \Cref{prop:alg_Koszul} and \Cref{prop:Lie_Koszul}, respectively.
  
  When $A$ and $L$ are Koszul dual, we write $L^! = A$ and $A^! = L$.
\end{definition}

The following is the more traditional way of introducing Koszul duals, which justifies the notations $L^!$ and $A^!$.
It is a straightforward reformulation of the preceding definition (using \Cref{rem:quad_alg_unique,rem:quad_lie_unqiue}).

\begin{proposition}
  Let $A$ be a graded commutative non-unital algebra and $L$ a graded Lie algebra.
  The following are equivalent:
  \begin{enumerate}
    \item $A$ and $L$ are Koszul dual.
    \item For any quadratic presentation $\pi \colon \SA V \to A$ of $A$ we have $L \iso \freeLie (\shift \dual V) / (S)$ where $S$ is the orthogonal complement of $R = \ker (\restrict \pi {\SA[2] V})$ under the pairing \eqref{eq:weight_2_pairing} associated to the pairing $V \tensor \shift \dual V \to \QQ$ of degree $-1$.
    \item For any quadratic presentation $\pi \colon \freeLie W \to L$ of $L$ we have that $A \iso \SA (\shift \dual W) / (R)$ where $R$ is the orthogonal complement of $S = \ker (\restrict \pi {\freeLie[2] W})$ under the pairing \eqref{eq:weight_2_pairing} associated to the pairing $\shift \dual W \tensor W \to \QQ$ of degree $-1$.
  \end{enumerate}
\end{proposition}

The following theorem is proven in \cite[Theorem 1.2 and Theorem 1.3]{Berglund14bis}.

\begin{theorem}[Berglund] \label{thm:formal_coformal}
  The following are equivalent for a simply connected space $X$:
  \begin{enumerate}
    \item $X$ is formal and coformal.
    \item $X$ is formal and $\rCoho * (X; \QQ)$ is a Koszul algebra.
    \item $X$ is coformal and $\htpygrp{* + 1}(X) \tensor \QQ$ is a Koszul Lie algebra.
  \end{enumerate}
  If these conditions are satisfied, then $\rCoho * (X; \QQ)$ and $\htpygrp{* + 1}(X) \tensor \QQ$ are Koszul dual via the restriction of the Hurewicz pairing
  \[
  \arraycolsep=0.1em
  \begin{array}{rclcl}
    \rCoho * (X; \QQ) & \tensor & \htpygrp{* + 1}(X) & \longto & \QQ \\
    x & \tensor & \alpha & \longmapsto & \pairing{x}{\operatorname{hur}(\alpha)}
  \end{array}
  \]
  to indecomposables.
  Here $\operatorname{hur}$ denotes the Hurewicz homomorphism.
\end{theorem}

We now give a number of examples to illustrate how \cref{thm:formal_coformal} can be used to make computations or to deduce non-formality results. Many of these are taken from \cite[\S5]{Berglund14bis}.

\begin{example}[$H$-spaces]
Let $X$ be a simply connected $H$-space (for example the loop space $\Loops Y$ of a $2$-connected space $Y$) and assume $X$ has finite $\QQ$-type.
Then the cohomology ring $\Coho{*}(X;\QQ)$ is a free graded commutative algebra (cf.~\cite[p.143]{FelixHalperinThomas01}), and hence it admits a quadratic presentation
    \[\Coho{*}(X;\QQ) \iso \SA V,\]
    with no non-trivial relations, for some graded vector space $V$. Free algebras are intrinsically formal and Koszul. By \cref{thm:formal_coformal}, $X$ is also coformal, and the homotopy Lie algebra is isomorphic to the Koszul dual of the cohomology ring. Thus,
    \[ \htpygrp{* + 1}(X) \tensor \QQ  \iso (\SA V)^! \iso \freeLie W / \big(\freeLie[2] W \big) \iso W,\]
    for $W = (sV)^\vee$.
    In particular, the Koszul duality between homotopy and cohomology in this case boils down to the statement that the indecomposables of the cohomology ring of an $H$-space may be identified with the dual of $\htpygrp{*}(X) \tensor \QQ$.
\end{example}

\begin{example}[Co-$H$-spaces]
Let $X$ be a simply connected co-$H$-space, e.g.\ the suspension $\Susp Y$ of a connected space $Y$, and assume $X$ has finite $\QQ$-type. By \cite{Berstein61}, $X$ is rationally equivalent to a wedge of spheres. In particular, $X$ formal and coformal. Moreover, since all cup products in the reduced cohomology of a co-$H$-space are trivial, the cohomology ring admits the quadratic presentation
    \[ \Coho * (X; \QQ)  \iso \SA V / ( \SA[2] V ) \]
    where $V=\rCoho * (X; \QQ)$. This is a Koszul algebra. The homotopy Lie algebra is isomorphic to the Koszul dual Lie algebra, which is free in this case;
    \[ \htpygrp{* + 1}(X) \tensor \QQ \iso  \left(\SA V / ( \SA[2] V ) \right)^! \iso \freeLie W \]
    where $W = (sV)^\vee \iso \rHo{*+1}(X;\QQ)$.
\end{example}

\begin{example}[Highly connected manifolds]
Let $n\geq 2$ and suppose that $M$ is an $(n-1)$-connected closed oriented manifold of dimension $\le 3n - 2$.  Neisendorfer--Miller \cite{NeisendorferMiller78} observed that such manifolds are both formal and coformal provided $\dim \Coho * (M;\QQ)\geq 4$. A description of the homotopy Lie algebra was stated in \cite[\S5]{Neisendorfer79}. We will here revisit these results from the point of view of higher structures, following \cite[\S4.1]{BerglundBorjeson17}.

Denote by $m_k$ the operations of the $\Cinfty$-algebra structure on $\Coho * (M; \QQ)$.
Suppose that $x_1,\ldots,x_k$ are non-trivial cohomology classes of positive degree. By the connectivity assumption, we must have $\deg {x_i}\geq n$ for all $i$. If $k\geq 3$, this implies 
\begin{align*}
    \deg {m_k(x_1, \dots, x_k)}  & =  \deg {x_1} + \dots + \deg {x_k} + 2 - k \\ & \ge  nk + 2 - k \\ & \ge  3(n-1) + 2 = 3n - 1,
    \end{align*}
    so that $m_k = 0$ for $k\neq 2$.
    By \cref{prop:formality conditions}, this implies that $M$ is formal.
    
By Poincaré duality the cohomology of $M$ is of the form $\QQ 1 \oplus V \oplus \QQ z$ where $V$ consists of  indecomposable cohomology classes and $z$ is a fundamental cohomology class.
When $\dim V > 1$, one can deduce that $\Coho{*}(M; \QQ)$ is Koszul (cf.\ \cite[Theorem 4.2]{BerglundBorjeson17}) and admits the quadratic presentation
$$\Coho * (M;\QQ) \iso \SA V/(R),$$
where $R$ is the kernel of the map induced by the cup product pairing $\mu(x\wedge y) = \langle x \cup y, [M] \rangle$:
\[ 0 \longto R \longto \SA[2] V \xlongto{\mu} \QQ \longto 0. \]
Hence, by \cref{thm:formal_coformal}, $M$ is coformal and the homotopy Lie algebra admits the presentation
$$\htpygrp{*+1}(M) \tensor \QQ \iso \freeLie (W) / (\omega),$$
where $W = \dual{(sV)}$ and $\omega \in \LL^2(W)$ is dual to $\mu$.

\end{example}   

\begin{example}[Symplectic manifolds]
As pointed out above, symplectic manifolds need not be formal, but Lupton--Oprea \cite[Corollary 2.7]{LuptonOprea94} show that every coformal simply connected compact symplectic manifold $M$ is formal. Moreover, the cohomology ring $\Coho * (M;\QQ)$ is generated in degree $2$ in this case.
    This result also shows that non-formal symplectic manifolds are examples of spaces that are neither formal nor coformal.
\end{example}

\begin{example}[Moduli spaces of genus $0$ curves]
The Deligne--Mumford compactification $\overline{\mathscr M}_{0,n}$ of the moduli space of genus $0$ curves with $n$ marked points is known to be formal, and Dotsenko has shown that $\Coho * (\overline{\mathscr M}_{0,n}; \QQ)$ is Koszul, see \cite{Dotsenko22}.
\end{example}

\begin{example}[Configuration spaces]
Consider the configuration space of $n$ (ordered) points in $\RR^m$,
    \[ \Conf n {\RR^m}  \defeq  \left\{ (x_1, \dots, x_n) \in (\RR^m)^n  \mid  x_i \neq x_j \text{ when } i \neq j \right\}. \]
    Let $a_{i,j} \in \Coho{m-1}\big(\Conf n {\RR^m};\QQ\big)$ be the class obtained by pulling back the fundamental cohomology class of $S^{m-1}$ along the map
    $$\xi_{i,j} \colon \Conf n {\RR^m} \longto \Sphere{m-1}$$
    that sends a point $(x_1,\ldots,x_n)$ to the unit vector $\frac{x_i - x_j}{|x_i - x_j|}$.
    The classes $a_{i,j}$ for $1 \leq i < j \leq n$ generate the cohomology ring, and the kernel $I$ of the map
    $$\SA (a_{i,j}) \longto \Coho{*}\big(\Conf n {\RR^m};\QQ\big)$$
    is generated by the so-called \emph{Arnold relations}:
    \begin{align*}
      a_{i,j}^2 &= 0, \\
      a_{i,j} a_{j,k} + a_{j,k} a_{k,i} + a_{k,i} a_{i,j} &= 0,
    \end{align*}
    for $i, j, k$ distinct (here we set $a_{i,j} = (-1)^m a_{j,i}$ if $i > j$); this is due to Arnold \cite{Arnold69} for $m = 2$ and Cohen \cite[Theorem 11.7]{Cohen76} in the general case.
    This algebra can be seen to be Koszul (cf.\ \cite[Theorem 1.2]{Bezrukavnikov94} in the case $m = 2$, or \cite[Example 5.5]{Berglund14bis}).
    Next, it is well-known that the space $\Conf n {\RR^m}$ is formal. For $m=2$ this was observed in \cite{Arnold69}. For $m>2$, this statement is contained in the much more general result \cite[Theorem 1.2]{LambrechtsVolic14}.
    
    For $m>2$, it follows that $\Conf n {\RR^m}$ is coformal and that its homotopy Lie algebra may be computed as the Koszul dual Lie algebra. One calculates that the Koszul dual is given by
        \[ \htpygrp{* + 1}\big(\Conf n {\RR^m}\big) \tensor \QQ  \iso  \freeLie(\alpha_{i,j}) / (S) \]
    for certain classes $\alpha_{i,j}$ of degree $m-2$ for $1 \le i < j \le n$ and $(S)$ the ideal generated by
    \begin{align*}
      [\alpha_{i,j}, \alpha_{k,l}] &, \qquad \text{for distinct $i,j,k,l$,} \\
      [\alpha_{i,j}, \alpha_{i,k} + \alpha_{j,k}] &,  \qquad \text{for distinct $i,j,k$.}
    \end{align*}
    Here we interpret $\alpha_{j,i} = (-1)^m \alpha_{i,j}$ if $i<j$.
    This is the \emph{Drinfeld--Kohno Lie algebra}.
    In other words, the Arnold relations are orthogonal to the Drinfeld--Kohno relations.
    A computation of the homotopy Lie algebra using different methods can be found in \cite{CohenGitler02}.
\end{example}

\begin{remark} \label{rem:generalization_non_sc}
If $X$ is a connected, not necessarily simply connected, space of finite $\QQ$-type, then one can show that \Cref{thm:formal_coformal} remains true upon replacing $\htpygrp{* + 1}(X) \tensor \QQ$ by $\htpygrp{* + 1} (\BKQcompl X)$, where $\BKQcompl X$ denotes the Bousfield--Kan $\QQ$-completion of $X$ (see Appendix \ref{sec:nonnilpotent rht}). Quillen's functor $\lambda X$ is only defined for simply connected spaces, but an appropriate notion of coformality is obtained by replacing $\lambda X$ in \cref{def:(co)formality} with the complete dg Lie algebra $\coHarr * (\rCoho * (X; \QQ))$, the Harrison cochains on the cohomology $\Cinfty$-algebra. This agrees with the notion of coformality discussed in \cite[\S10.3]{BFMT20}.
\end{remark}


A special case of \cref{thm:formal_coformal} is the following result of Papadima--Yuzvinsky \cite{PapadimaYuzvinski99}. This was observed in \cite[Corollary 1.9]{Berglund14bis} in the nilpotent case, but by using \cref{rem:generalization_non_sc} one can upgrade this observation to cover the general case.

\begin{theorem}[Papadima--Yuzvinsky] \label{thm:PY}
  Let $X$ be a connected space of finite type.
  If $X$ is formal, then $X$ is a rational $K(\pi,1)$-space if and only if the cohomology ring $\Coho{*}(X;\QQ)$ is Koszul with weight equal to cohomological degree.
\end{theorem}

\begin{example}
The space $\Conf n {\RR^2}$ is not simply connected. In fact, it is a model for the classifying space of the pure braid group on $n$ strands.
Formality and Koszulness of the cohomology ring then imply that the Bousfield--Kan $\QQ$-completion of $\Conf n {\RR^2}$ is a rational $K(\pi,1)$ and that its fundamental group, which may be identified with the Malcev completion of the pure braid group, is the Malcev group associated to the completion of the Drinfeld--Kohno Lie algebra.
See \cref{rem:generalization_non_sc,thm:PY} as well as \cite[\S8.6]{FelixHalperinThomas15}.
\end{example}

\begin{remark}
      Salvatore \cite{Salvatore20} has shown that $\Conf n {\RR^m}$ is intrinsically formal over any commutative ring provided $n\leq m$, but $\Conf n {\RR^2}$ is not formal over $\F 2$ when $n \ge 4$.
\end{remark}

\subsection{Koszul \texorpdfstring{$\Cinfty$}{C∞}- and \texorpdfstring{$\Linfty$}{L∞}-algebras}

Berglund--Börjeson \cite{BerglundBorjeson20} extended the notion of Koszul algebras to $A_\infty$-algebras and linked it to formality. The following is an adaptation of the results of \cite{BerglundBorjeson20} to $C_\infty$- and $L_\infty$-algebras. It leads to a generalization of the results of \cite{Berglund14bis} discussed in the previous section.

\begin{definition} \label{def:Cinfty_Linfty_Koszul}
  A $\Cinfty$-algebra $A$ is \emph{Koszul} if it is equipped with a weight grading
  \[ A  =  \Dirsum_{w = 1}^\infty \weight w A \]
  such that the structure operations $m_n$ are homogeneous of weight $2 - n$ and such that the cochain complex
  \[ \weight 0 {\Harr * (A)}  \xlongto{b}  \weight 1 {\Harr * (A)}  \xlongto{b}  \dots \]
  is exact, where $\Harr * (A)  =  (\cofreeLie(\shift A), b)$ is weight graded by letting $\shift$ have weight $-1$.
  
  Similarly an $\Linfty$-algebra $L$ is \emph{Koszul} if it is equipped with a weight grading
  \[ L  =  \Dirsum_{w = 1}^\infty \weight w L \]
  such that the structure operations $l_n$ are homogeneous of weight $2-n$ and such that the cochain complex
  \[ \weight 0 {\CE * (L)}  \xlongto{b}  \weight 1 {\CE * (L)}  \xlongto{b}  \dots \]
  is exact, where $\CE * (L) = (\coSA(\shift L), b)$ is weight graded by letting $\shift$ have weight $-1$.
  
  We consider these weight gradings to be cohomological (note that they are, accordingly, written as superscripts).
\end{definition}

In contrast to the preceding subsection, here we allow $A$ and $L$ to have a non-trivial differential.
Basic examples of Koszul $\Linfty$-algebras are provided by free $\Linfty$-algebras.

\begin{example}
  Let $W$ be a graded vector space.
  There is a weight grading on the free $\Linfty$-algebra $\freeLinfty (W)$ determined by letting $l_n$ have weight $2 - n$ and putting $W$ in weight $1$.
  With this weight grading $\freeLinfty (W)$ is Koszul.
\end{example}

\begin{proposition}
Every minimal Koszul $\Linfty$-algebra $L$ is generated in weight $1$, i.e., letting $W= \weight 1 L$, the canonical (strict) morphism of $\Linfty$-algebras
$$\freeLinfty(W) \longto L$$
is surjective.
\end{proposition}

\begin{proof}
By definition of Koszulness, there is a quasi-isomorphism of dg coalgebras $C\to \CE *(L)$, where $C$ denotes the weight zero homology of $\CE *(L)$. We then obtain weight homogeneous strict quasi-isomorphisms of $\Linfty$-algebras
$$\coHarrinfty * (C) \xlongto{\eq} \coHarrinfty * (\CE * (L)) \xlongto{\eq} L.$$
Since $L$ has trivial differential, the composite map must be surjective. Since $\coHarrinfty * (C) = \freeLinfty (s^{-1} C)$ is generated in weight $1$, this shows that $L$ is generated in weight $1$.
\end{proof}

We now turn to the definition of the Koszul dual commutative algebra of a Koszul $\Linfty$-algebra.


\begin{lemma} \label{lemma:freeLinfty_dual}
  If $\shift W$ is dual to $V$, then $\weight 2 {\freeLinfty (W)}$ is dual to $\shift[2] \SA V$.
\end{lemma}

\begin{proof}
Note that $\weight 2 {\freeLinfty (W)}$ is spanned by terms of the form $l_n(w_1, \dots, w_n)$, so we may identify
$$\weight 2 {\freeLinfty (W)} = \Dirsum_{n \ge 1} E(n) \tensor[\Symm{n}] W^{\otimes n},$$
where $E(n)$ is the sign representation of $\Symm{n}$ concentrated in degree $n-2$. The right hand side can be rewritten as
$s^{-2} \Lambda sW$, which is dual to $s^2 \Lambda V$.
\end{proof}


\begin{definition} \label{def:Linfty koszul dual}
  Let $L$ be a Koszul $\Linfty$-algebra, let $W= \weight 1 L$ and consider the short exact sequence in weight 2
  \[ 0  \longto  S  \longto  \weight 2 {\freeLinfty (W)}  \longto  \weight 2 L  \longto  0 \]
  arising from the surjective map $\freeLinfty (W) \to L$.
  We let $V$ be dual to $\shift W$ and denote by $\orth S \subseteq \SA V$ the linear subspace obtained (up to shifts) as the orthogonal complement of $S$ under the pairing of \Cref{lemma:freeLinfty_dual}.
  We define the \emph{Koszul dual algebra of $L$} to be the algebra 
  $$L^! = \SA V / (\orth S)$$
  i.e.\ the quotient of $\SA V$ by the ideal generated by $\orth S$.
\end{definition}

The Koszul dual Lie algebra of a Koszul $\Cinfty$-algebra is defined similarly.

\begin{definition} \label{def:Cinfty koszul dual}
Let $A$ be a Koszul $\Cinfty$-algebra and let $V= \weight 1 A$. The \emph{Koszul dual Lie algebra} is defined by
$$A^! = \freeLie(W)/(\orth R),$$
where $W$ is dual to $sV$ and where $\orth R$ is the orthogonal complement to the kernel $R$ in the exact sequence
\[ 0  \longto  R  \longto  \weight 2 {\freeCinfty (V)}  \longto  \weight 2 A  \longto  0 \]
of graded vector spaces.
\end{definition}

If $L$ is a Koszul $\Linfty$-algebra, then the Koszul dual commutative algebra $L^!$ is exactly the weight zero cohomology of the weight graded cochain algebra $\coCE{*}(L)$, and there is an exact sequence
  \[ \dots  \longto  \coCE{*}(L)_{(1)}  \longto  \coCE{*}(L)_{(0)}  \longto  L^! \longto  0.\]

Similarly, if $A$ is a Koszul $\Cinfty$-algebra, then the Koszul dual Lie algebra $A^!$ is the weight zero cohomology of $\coHarr{*}(A)$.

\begin{theorem}
  Let $X$ be a simply connected space of finite $\QQ$-type.
  \begin{enumerate}
    \item $X$ is formal if and only if the $\Linfty$-algebra $\htpygrp{* + 1}(X) \tensor \QQ$ is Koszul.
    In this case $\rCoho{*}(X; \QQ)$ is the Koszul dual algebra.
    \item $X$ is coformal if and only if the $\Cinfty$-algebra $\rCoho{*}(X; \QQ)$ is Koszul.
    In this case $\htpygrp{* + 1}(X) \tensor \QQ$ is the Koszul dual Lie algebra.
  \end{enumerate}
\end{theorem}

\begin{proof}
See the proof of \cite[Theorem 2.9]{BerglundBorjeson20}.
\end{proof}

\begin{remark}
  That $L = \htpygrp{* + 1}(X)\tensor \QQ$ is Koszul implies that the minimal model $\coCE{*}(L)$ has an extra homological grading such that the chain complex
  \begin{equation} \label{eq:bigraded model}
  \dots  \longto  \coCE{*}(L)_{(1)}  \longto  \coCE{*}(L)_{(0)}  \longto  \rCoho * (X; \QQ)  \longto  0
  \end{equation}
  is exact.
  This recovers the ``bigraded model'' of Halperin--Stasheff \cite{HalperinStasheff79}.
\end{remark}






\begin{example} \label{ex:noncoformal} 
Complex projective space $\CP{n}$ is formal, for example since it is a compact Kähler manifold, but the cohomology algebra
    $$\Coho * (\CP n; \QQ) \iso \QQ[x] / (x^{n+1})$$
    is not Koszul for $n>1$, so $\CP{n}$ is not coformal by \cref{thm:formal_coformal}.
    The homotopy $\Linfty$-algebra may be presented as
    $$\htpygrp{* + 1}(\CP n) \tensor \QQ  \iso  \linspan{\alpha, \beta},$$
    where $|\alpha|=1$, $|\beta| = 2n$ and the only non-trivial operation is given by
    $$l_{n+1}(\alpha, \dots, \alpha) = (n+1)!\beta.$$
    This can be seen by observing that the minimal Sullivan model is given by
    $$
\big(\Lambda (x,y), d \big),\quad dx= 0, \,\, dy= x^{n+1},\quad |x|=2,\,\, |y| = 2n+1,
$$
  and then taking the corresponding $\Linfty$-algebra as in \cref{prop:L-infinity dictionary}. Alternatively, one can use the homotopy fiber sequence
  $$\Sphere 1 \longto \Sphere{2n+1} \longto \CP{n}$$
  to argue that $\pi_*(\CP{n})\otimes \QQ$ is spanned by the homotopy classes of the evident maps
    \begin{align*}
      \alpha &\colon \Sphere 2 \iso \CP 1 \longto \CP n, \\
      \pi &\colon \Sphere {2n+1} \longto \CP n,
    \end{align*}
    and then observe, following Porter \cite{Porter67}, that $(n+1)!\pi$ may be identified with the Whitehead product $\langle \alpha, \dots, \alpha \rangle$ of order $n+1$. Using the connection between higher order Whitehead products and the $\Linfty$-algebra structure on $\htpygrp{*+1}(X)\otimes \QQ$ (see \cref{prop:whitehead vs Linfty}), this implies the result.
    
    The homotopy $\Linfty$-algebra is Koszul if $\alpha$ and $\beta$ are assigned weight $1$ and $2$, respectively. The sequence \eqref{eq:bigraded model} assumes the form
    $$0 \longto \QQ[x] y \xlongto{d} \QQ[x] \longto \QQ[x]/(x^{n+1}) \longto 0$$
    in this case.
    
    The Koszul duality between the homotopy $\Linfty$-algebra and the cohomology algebra, as expressed in \cref{def:Linfty koszul dual}, takes the following form: in weight $2$, we have the exact sequence
    \[ 0  \longto  S \longto \weight 2 {\freeLinfty (\alpha)}  \longto  \linspan \beta  \longto  0 \]
    where $S$ is spanned by $l_k(\alpha, \dots, \alpha)$ for all $k\geq 2$ except $k=n+1$.
    Since we have that
    $$\langle l_k(\alpha,\ldots,\alpha),x^m \rangle \ne 0$$
    if and only if $k=m$, we see that $\orth S = \linspan{x^{n+1}}$ and thus
    $$\Coho * (\CP n; \QQ)\iso \big(\htpygrp{*+1}(\CP{n})\otimes \QQ \big)^! = \QQ[x] / (x^{n+1}),$$
    as expected.
\end{example}    
    
\begin{example} \label{ex:nonformal}    
Let $n\geq 2$ and let
    \begin{equation} \label{eq:sphere bundle}
    \Sphere{2n-1} \longrightarrow M \longrightarrow \Sphere n \times \Sphere n
    \end{equation}
    be a smooth sphere bundle with non-trivial Euler class
    $$0\ne e\in \Coho{2n}(\Sphere n \times \Sphere n;\QQ).$$
    For example, we can take $M$ to be the sphere bundle associated to the pullback $c^*(\mathrm T \Sphere {2n})$ of the tangent bundle of $\Sphere{2n}$ along the collapse map
    $$c\colon \Sphere {n} \times \Sphere {n} \longto \Sphere{2n}.$$
    The total space $M$ is an $(n-1)$-connected closed manifold of dimension $4n-1$ and, as we will see, it is coformal but not formal.


    The long exact sequence of rational homotopy groups associated to the fibration \eqref{eq:sphere bundle} is easily seen to split, which yields a decomposition
    \begin{align*}
      \htpygrp{* + 1}(M) \tensor \QQ  & \iso \big(\htpygrp{* + 1}(\Sphere{2n-1})\tensor \QQ\big) \oplus \big(\htpygrp{*+1}(\Sphere{n}\times \Sphere{n})\tensor \QQ\big).
    \end{align*}
    The left summand is one-dimensional, spanned by a class $\gamma$ in degree $2n-2$, and the right summand is spanned by classes $\alpha$ and $\beta$ in degree $n-1$ and, if $n$ is even, by $[\alpha,\alpha]$ and $[\beta,\beta]$ in degree $2n-2$. There is no room for non-trivial $\Linfty$-operations $l_k$ for $k\geq 3$, so $M$ is necessarily coformal. The only possible non-trivial Lie bracket not already accounted for is $[\alpha,\beta]$. Since it vanishes in $\pi_{*+1}(S^n \times S^n)$ we must have $[\alpha,\beta] = a\gamma$ for some $a\in \QQ$ and one can check that $a\ne 0$ since the Euler class is non-trivial. This implies that the homotopy Lie algebra admits the following cubic presentation:
    $$\htpygrp{*+1}(M)\tensor \QQ \cong \LL(\alpha,\beta)/\left([\alpha,[\alpha,\beta]], [\beta,[\alpha,\beta]]\right).
    $$
    It is easily seen that this graded Lie algebra does not admit any quadratic presentation, so it cannot be Koszul. Since $M$ is coformal, this implies that $M$ is not formal by \cref{thm:formal_coformal}.
    
    We obtain the minimal Sullivan model by applying Chevalley--Eilenberg co\-chains. For a suitable choice of generators, it assumes the form
    $$\big(\Lambda(x,y,z),d\big),\quad dx=dy = 0,\,\, dz =xy,$$
    where $|x|=|y| = n$ and $|z|=2n-1$, if $n$ is odd. If $n$ is even, it assumes the form
    $$\big(\Lambda(x,y,s,t,z),d\big),\quad dx=dy = 0,\,\, ds=x^2,\,\, dt=y^2,\,\,dz =xy,$$
    where $|x|=|y| = n$ and $|s|=|t|=|z|=2n-1$.
    
    The cohomology $\Cinfty$-algebra of $M$ is easily computed from the Sullivan model by using the homotopy transfer theorem (cf.~\cite[Corollary 4.12]{BerglundBorjeson20}). For a suitable choice of basis it is given by
    $$\Coho * (M; \QQ) \iso \linspan{x,y,u,v,w},$$
    where $|x|=|y|=n$, $|u|=|v|=3n-1$, $|w|=4n-1$, and the only non-trivial $\Cinfty$-operations are given by
    \begin{gather*}
    xv = yu = w, \\
    m_3(x,x,y) = - m_3(y,x,x) = - u, \\
    m_3(x,y,y) = -m_3(x,y,y) = v.
    \end{gather*}
    This is a Koszul $\Cinfty$-algebra if $x,y$ are assigned weight $1$, the classes $u,v$ are assigned weight $2$ and $w$ is assigned weight $3$.
\end{example}


\section{Lecture 3: Graph complexes and automorphisms of manifolds}

In this lecture, we will discuss another type of higher structure in rational homotopy theory. It arises in the study of the rational cohomology of automorphisms of high dimensional manifolds and was discovered by Berglund--Madsen \cite{BerglundMadsen20}. The higher structure in question is the homology of a certain graph complex in the sense of Kontsevich \cite{Kontsevich93,Kontsevich94}.

We begin by reviewing a few results, some classical and some more recent, on the rational homotopy theory of classifying spaces of homotopy automorphisms of simply connected CW-complexes. Then we review some basics about \emph{modular operads}, which is an efficient tool for handling graph complexes.
Finally, we discuss the results of Berglund--Madsen \cite{BerglundMadsen20}.

\subsection{Classifying spaces of homotopy automorphisms}
Let $X$ be a finite CW-complex and denote by
$$\aut(X)$$
the topological monoid of homotopy equivalences from $X$ to itself, equipped with the compact-open topology. Let $\B \aut(X)$ be its classifying space, defined, for example, using the geometric bar construction \cite{May75}.
This space classifies fibrations with fiber $X$. More precisely, for every CW-complex $B$, there is a natural bijection
\[ \big[B, \B \aut(X)\big]  \iso  \left\{ \pi \colon E \to B \text{ fibration} \mid \inv\pi(b) \eq X \text{ for all } b \in B \right\} / {\sim} \]
where the right-hand side denotes the set of fiber homotopy equivalence classes of fibrations over $B$ whose fiber is homotopy equivalent to $X$, cf.~\cite{Stasheff63bis,May75}. 
In particular, this means that the cohomology ring $\Coho * (\B \aut(X))$ may be identified with the ring of characteristic classes of fibrations with fiber $X$. It is a fundamental problem in homotopy theory to compute this ring.




To compute it rationally, one may attempt to use tools from rational homotopy theory. An issue is that the space $\B \aut(X)$ is in general not nilpotent, so the standard methods are not directly applicable, but one can study the homotopy fiber sequence
\[ \trunc 1 {\B \aut(X)}  \longto  \B \aut(X)  \longto  \B \heq{X}.\]
Here, the base is the classifying space of the discrete group
$$\heq{X} \defeq \htpygrp{1} (\B\aut(X)) = \htpygrp{0} (\aut(X))$$
of homotopy classes of self--homotopy equivalences of the space $X$.
The fiber $\trunc 1 {\B \aut(X)}$ is the $1$-connected cover of $\B \aut(X)$, i.e.\ the unique (up to weak homotopy equivalence) simply connected space admitting a map to $\B \aut(X)$ that induces an isomorphism on $\htpygrp{k}(-)$ for $k > 1$.\footnote{The universal cover, when it exists, is a model for the $1$-connected cover.} Since $\trunc 1 {\B \aut(X)}$ is simply connected, the standard methods of rational homotopy theory are in principle applicable, and in fact tractable Lie models for this space can be written down (see \cref{thm:Baut model} below).

The following result, due independently to Sullivan \cite{Sullivan77} and Wilkerson \cite{Wilkerson76}, puts strong constraints on the group $\heq{X}$ when $X$ is a simply connected finite CW-complex. This was one of the primary applications of Sullivan's rational homotopy theory.

\begin{theorem}[Sullivan, Wilkerson] \label{thm:Sullivan-Wilkerson}
If $X$ is a simply connected finite CW-complex, then $\heq{X}$ is an arithmetic group.
\end{theorem}

Typical examples of arithmetic groups are the $\ZZ$-points in linear algebraic groups defined over $\QQ$, such as $\mathrm{GL}_n(\ZZ)$ or $\mathrm{SL}_n(\ZZ)$. See \cite{Serre79} for an introduction.

\begin{example} \label{ex:heq of wedge}
It is an exercise\footnote{In fact, this is \cite[\S 4.A, Exercises 3 and 4]{Hatcher02}.} to check that
  \[
  \heq{\textstyle \bigvee^n \Sphere k}
  \iso
  \begin{cases}
    \Out(\freegroup{n}), & k = 1, \\
    \mathrm{GL}_n(\ZZ), & k > 1.
  \end{cases}
  \]
  The group $\mathrm{Out}(\freegroup{n})$ of outer automorphisms of a free group on $n$ generators is finitely presented, but it is not arithmetic when $n \ge 3$ (cf.~\cite[\S2.8.1]{Vogtmann02}).
\end{example}


\begin{remark}
Arithmetic groups $\Gamma$ are known to be of \emph {finite type} in the sense that $\B\Gamma$ has the homotopy type of a CW-complex with finitely many cells in each dimension. Groups of finite type are in particular finitely presented.

Dror--Dwyer--Kan have shown that $\heq{X}$ is of finite type whenever $X$ is virtually nilpotent \cite{DrorDwyerKan81} (they do not establish the stronger property of arithmeticity in this generality). If $X$ is not virtually nilpotent, then $\heq{X}$ is not necessarily of finite type. An example where $\heq{X}$ is not even finitely generated is given by
  $$X = S^1 \vee S^2 \vee S^3,$$
  see \cite{FrankKahn77}.
\end{remark}

Rational models for the simply connected cover of $\B\aut(X)$ are known.
We here review a version for $\B\paut(X)$, where $\aut_*(X)$ denotes the topological monoid of pointed self--homotopy equivalences of $X$.

For a dg Lie algebra $L$ with differential $\delta$, we let $\Der (L)$ denote the dg Lie algebra of derivations of $L$. Its elements of degree $n$ are maps $\theta\colon L \to L$ of degree $n$ that satisfy
$$\theta[x,y] = [\theta(x),y] + (-1)^{n|x|}[x,\theta(y)]$$
for all $x,y\in L$. The commutator
$$[\theta,\eta] = \theta \circ \eta - (-1)^{|\theta||\eta|}\eta \circ \theta$$
makes $\Der (L)$ into a graded Lie algebra, and equipped with the differential $[\delta,-]$, it becomes a dg Lie algebra. The truncation $\trunc 1 {\Der (L)}$ is defined as the subalgebra which is zero in non-positive degrees, agrees with $\Der (L)$ in degrees $>1$ and consists of the derivations $\theta$ such that $[\delta,\theta] = 0$ in degree $1$.

The following is due to Tanré \cite[Corollaire VII.4.(4)]{Tanre83} and Schlessinger--Stasheff \cite{SchlessingerStasheff12}.

\begin{theorem} \label{thm:Baut model}
Let $X$ be a simply connected finite CW-complex.
  The simply connected cover $\trunc 1 {\B \paut(X)}$ has Lie model $\trunc 1 {\Der (\LL_X)}$, where $\LL_X$ is the minimal Quillen model of $X$.
\end{theorem}


Let us briefly mention some recent generalizations and extensions of this result:
  \begin{itemize}
      \item Berglund--Saleh \cite{BerglundSaleh20} construct Lie models for $\trunc 1 {\B \aut_A(X)}$, where $\aut_A(X)$ is the topological monoid of homotopy automorphisms of $X$ that restrict to the identity on a given non-empty subcomplex $A \subseteq X$.
    \item Félix--Fuentes--Murillo \cite{FFM21} construct  Lie models for nilpotent covers $\B \aut_G(X) \to \B \aut(X)$ associated to subgroups $G\subseteq \pi_0(\aut_*(X))$ that act nilpotently on $\Ho{*}(X)$ (that such covers are nilpotent is a result due to Dror--Zabrodsky \cite{DrorZabrodsky79}).
    
    \item Berglund--Zeman \cite{BerglundZeman21} incorporate the action of the deck transformation group $\Gamma$ into Lie models $\lie g$ for certain nilpotent covers of the form $\B \aut_G(X)$ and use this to express the rational homotopy type of $\B \aut(X)$ as a homotopy orbit space $\MC[\bullet](\lie g)_{h \Gamma}$.
    
  \end{itemize}
Next, let $M$ be a simply connected compact $m$-dimensional manifold with boundary $\bdry M = \Sphere {m-1}$. We will let
$$\bdryaut(M)$$
denote the topological monoid of self--homotopy equivalences of $M$ that restrict to the identity on the boundary. Fix a minimal Quillen model $\LL_M$ for $M$. There is a distinguished cycle $\omega \in \LL_M$ of degree $m-2$ that represents the inclusion $\Sphere {m-1} = \bdry M \to M$ under the isomorphism
$$\htpygrp{m-1}(M) \tensor \QQ \iso \Ho{m-2}(\LL_M).$$
Let $\Der[\omega] (\LL_M)$ denote the dg Lie algebra of derivations $\theta \colon \LL_M \to  \LL_M$ such that $\theta(\omega) = 0$.
Note that $\Der[\omega] (\LL_M)$ is closed under the differential $[\delta,-]$ since $\delta(\omega) = 0$.
The following theorem is proven in \cite[Theorem 3.12]{BerglundMadsen20}. (It simplifies the model provided by \cite{BerglundSaleh20}.)

\begin{theorem}[Berglund--Madsen] \label{thm:bdryaut_Lie_model}
Let $M$ be a simply connected compact $m$-manifold with boundary $\partial M = S^{m-1}$. The simply connected cover $\trunc 1 {\B \bdryaut(M)}$ has Lie model $\trunc 1 {\Der[\omega] (\LL_M)}$.

\end{theorem}


\begin{example} \label{ex:W_g}
  Consider the $2d$-dimensional manifold
  $$W_g = \Conn_g \bigl( \Sphere d \times \Sphere d \bigr) $$
  and let $W_{g,1}$ denote the result of removing the interior of an embedded disk $\Disk{2g} \subset W_g$.
  For $d = 1$, the manifold $W_{g,1}$ is an orientable surface of genus $g$ and one boundary component.
  Now assume that $d > 1$.
  Then $W_{g,1} \eq \bigvee^{2g} \Sphere {d}$ has Lie model $\freeLie V_g$ with trivial differential $\delta = 0$, where $V_g$ is the graded vector space $s^{-1}\rHo{*}(W_{g,1})$, which is concentrated in degree $d-1$.
  For a suitable choice of basis $\alpha_1, \beta_1, \dots, \alpha_g, \beta_g$ for $V_g$, the element $\omega$ assumes the form
  $$\omega = [\alpha_1, \beta_1] + \dots + [\alpha_g, \beta_g].$$
  It follows that $\trunc 1 {\B \bdryaut(W_{g,1})}$ has Lie model 
  $\trunc 1 {\Der[\omega] (\LL V_g)}$. The differential $[\delta,-]$ is trivial, so in particular $\trunc 1 {\B \bdryaut(W_{g,1})}$ is coformal.
\end{example}

\subsection{Graph complexes and modular operads} 
Graph complexes, invented by Kontsevich \cite{Kontsevich93}, were a major impetus for the development of the theory of algebraic operads and Koszul duality for operads.
A spectacular application is that the cohomology of various groups of central importance to geometric topology, such as mapping class groups of surfaces or automorphism groups of free groups, can be expressed in terms of graph homology (see \cref{thm:kckv} below for a precise statement of one such result). However, the computation of graph homology is a difficult problem the complete solution of which remains an open problem.

Modular operads and the Feynman transform of Getzler--Kapranov \cite{GetzlerKapranov98} provide a convenient framework for handling Kontsevich's graph complexes and generalizations thereof. The Feynman transform may be thought of as a (dualized) bar construction for modular operads. To define modular operads, we first need to review cyclic operads \cite{GetzlerKapranov95}.

\begin{definition}
  A \emph{cyclic operad} is an operad $\operad P$ together with an extension of the action of $\Symm n$ on $\operad P (n)$ to an action of $\Symm{n+} \defeq \Bij(\{0, \dots, n\})$ (one should think of this as also permuting the output with the inputs) such that
  \[ (a \circ_m b) \act t_{m + n - 1}  =  (a \act t_m) \circ_1 (b \act t_n) \]
  holds for all $a \in \operad P(m)$ and $b \in \operad P(n)$.
  Here $t_k$ denotes the permutation $(0 1 \dots k) \in \Symm{k+}$ and $c \circ_l d$ is the result of piping the output of $d$ into the $l$-th input of $c$.
  
  For a cyclic operad $\operad P$ and $n \ge 0$, we set $\operad P \cyc{n + 1} \defeq \operad P (n)$ as a $\Symm{n+}$-module.
\end{definition}

\begin{remark}
  The condition in the definition of a cyclic operad is equivalent to requiring composition along \emph{un}rooted trees to be well-defined, analogously to operads and rooted trees.
\end{remark}

\begin{definition}
  A \emph{modular operad} is a cyclic operad $\operad P$ equipped with a grading
  \[ \operad P \cyc n  =  \Dirsum_{g = 0}^\infty \operad P \cyc{g, n} \]
  (called the \emph{genus} grading) and for each pair $i, j \in \{0, \dots, n\}$ with $i \neq j$ a contraction operation
  \[ \xi_{i,j} \colon \operad P \cyc{g, n} \longto \operad P \cyc{g+1, n-2} \]
  (which one should think of as connecting the $i$-th and $j$-th in/output to each other) such that composition along arbitrary graphs is well-defined.
  An explicit list of axioms equivalent to this condition can be found in \cite[Theorem 3.7]{GetzlerKapranov98}.
\end{definition}

\begin{remark}
  It is customary to impose the so called ``stability condition''
  \[ \operad P \cyc{g, n} \iso 0  \text{ when } 2g + n \le 2 \]
  in the definition of a modular operad.
  We will not need this restriction however, and so we drop it from the definition.
\end{remark}

\begin{example}
  The operads $\Lie$ and $\Cinfty$ are cyclic operads.
  Any cyclic operad can be considered as a modular operad by putting everything in genus $0$ and letting all contractions act trivially.
\end{example}

There is another way of producing a modular operad from $\Cinfty$ (or any other cyclic operad).
Namely we let $\operad F$ be the modular operad freely generated by the cyclic operad $\Cinfty$.
In genus $0$ this has $\operad F \cyc{0, n} \iso \Cinfty \cyc n$ and in higher genera it is freely generated from genus $0$ under applications of the contraction operations.
The modular operad $\operad F$ is the ``modular envelope'' of the cyclic operad $\Cinfty$ and isomorphic to the ``Feynman transform'' of the cyclic operad $\Lie$ considered as a modular operad, see \cite[Corollary 9.3]{Ward19}.

The differential of $\Cinfty$ induces a differential $\del$ on $\operad F$.
For example we have
\[ \del \colon \xi_{1,2}(m_3) = 
\begin{tikzpicture}[baseline = (m.base)]
  \draw (1,-1) node (m) {$m_3$};
  \draw (0,0) -- (m);
  \draw (1,0) -- (m);
  \draw (2,0) node[anchor=south west] {\tiny 1} -- (m);
  \draw (m.south) -- (1,-2) node[anchor=north] {\tiny 0};
  \draw (0,0) .. controls (0,1) and (1,1) .. (1,0);
\end{tikzpicture}
\longmapsto
\begin{tikzpicture}[baseline = (base.base)]
  \draw (1.5,-0.66) node (m1) {$m_2$};
  \draw (1,-1.33) node (m2) {$m_2$};
  \draw (1,-1) node (base) {};
  \draw (0,0) -- (m2);
  \draw (1,0) -- (m1);
  \draw (2,0) node[anchor=south west] {\tiny 1} -- (m1);
  \draw (m1) -- (m2);
  \draw (m2) -- (1,-2) node[anchor=north] {\tiny 0};
  \draw (0,0) .. controls (0,1) and (1,1) .. (1,0);
\end{tikzpicture}
-
\begin{tikzpicture}[baseline = (base.base)]
  \draw (0.5,-0.66) node (m1) {$m_2$};
  \draw (1,-1.33) node (m2) {$m_2$};
  \draw (1,-1) node (base) {};
  \draw (0,0) -- (m1);
  \draw (1,0) -- (m1);
  \draw (2,0) node[anchor=south west] {\tiny 1} -- (m2);
  \draw (m1) -- (m2);
  \draw (m2) -- (1,-2) node[anchor=north] {\tiny 0};
  \draw (0,0) .. controls (0,1) and (1,1) .. (1,0);
\end{tikzpicture}
\]

in $\operad F \cyc{1,2}$.
Hence $\operad F \cyc {g,n}$ is a finite chain complex
\[ \operad F \cyc{g,n}_{2g-3+n}  \xlongto{\del}  \operad F \cyc{g,n}_{2g-4+n}  \xlongto{\del}  \dots  \xlongto{\del}  \operad F \cyc{g,n}_{0} \]
(which degrees can be non-trivial is an easy combinatorial consequence of $m_k$ having degree $k - 2$).
This is one example of a ``graph complex''.
In genus $0$ its homology is easy to describe, being given by
\[ \Ho k (\operad F \cyc {0,n}) \iso \begin{cases} \QQ, & k = 0 \\ 0, & k > 0 \end{cases} \]
but in higher genera this has a very interesting, and complicated, structure which is not yet fully understood.

The homology of the graph complex $\operad F$ turns out to be related to the homology of certain automorphism groups of free groups. As discussed above (see \cref{ex:heq of wedge}), the group of homotopy classes of self--homotopy equivalences of a bouquet of $g$ circles may be identified with the group of outer automorphisms of the free group on $g$ generators, 
$$\heq{\textstyle \bigvee^g S^1} \iso \Out(\freegroup{g}).$$
More generally, define $A_{g,n}$ to be the group of homotopy classes of self--homotopy equivalences of $\bigvee^g S^1$ relative to $n$ marked points. Then
$$
A_{g,0} \iso \Out(\freegroup{g}), \quad A_{g,1} \iso \Aut(\freegroup{g}), \quad A_{g,n} \iso \Aut(\freegroup{g}) \ltimes \freegroup{g}^{n-1}.
$$


The following is due to Kontsevich \cite{Kontsevich93,Kontsevich94} for $n = 0$ and Conant--Kassabov--Vogtmann \cite[Theorem 11.1]{CKV13} for $n>0$.

\begin{theorem}[Kontsevich, Conant--Kassabov--Vogtmann] \label{thm:kckv}
For all $g + n \ge 2$ and all $k$, there is an isomorphism
\[ \Ho k (\operad F \cyc {g,n})  \iso  \Ho k (A_{g,n}; \QQ). \]
\end{theorem}

The homology of either side is still largely unknown. See \cite{CKV13,CHKV16} for some computations and a discussion.


\subsection{Graph homology and automorphisms of manifolds}
We will now discuss the results of Berglund--Madsen \cite{BerglundMadsen20} that relate the cohomology of automorphisms of high dimensional manifolds to Kontsevich graph homology. Certain key parts of the proof \cite{BerglundMadsen20} have been simplified by \cite[\S4.2]{BerglundZeman21}, and the outline presented in this section will follow the latter.

Recall that 
$$W_g = \Conn_g \bigl( \Sphere{d} \times \Sphere{d} \bigr) $$
and that $W_{g,1}$ denotes what is left after removing the interior of a small embedded disk $\Disk{2d} \subset W_g$.

Let $\Gamma_g$ be the group of automorphisms of $\Ho{d}(W_{g})$ that are induced by a self--homotopy equivalance of $W_g$. Under the identification of the group of automorphisms of the abelian group $\Ho{d}(W_{g}) \cong \ZZ^{2g}$ with $\GL_{2g}(\ZZ)$, one can check that
$$\Gamma_g =
\begin{cases}
\OO_{g,g}(\ZZ), & \mbox{$d$ even}, \\
\Sp_{2g}(\ZZ), & d=1,3,7, \\
\Sp_{2g}^q(\ZZ), & \mbox{$d$ odd, $d\ne 1,3,7$},
\end{cases}$$
where $\Sp_{2g}^q(\ZZ) \subseteq \Sp_{2g}(\ZZ)$ is the finite-index subgroup of symplectic $(2g\times 2g)$-matrices
$$
\begin{pmatrix}
A & B \\ C & D
\end{pmatrix}
$$
such that the diagonal entries of the $(g\times g)$-matrices $C^t A$ and $D^t B$ are even. It can be shown that each element of $\Gamma_g$ can be realized by a self--homotopy equivalence (in fact diffeomorphism) of $W_{g,1}$ that fixes the boundary pointwise, cf.~\cite[\S5.2]{BerglundMadsen20}.

Next, one considers the homotopy fiber sequence
\begin{equation} \label{eq:hfib}
\B\mathrm{tor}_\partial(W_{g,1}) \longto \B\bdryaut(W_{g,1}) \longto \B\Gamma_g,
\end{equation}
where $\mathrm{tor}_\partial(W_{g,1}) \subset \bdryaut(W_{g,1}) $ denotes the ``Torelli submonoid'', i.e.\ the submonoid of self--homotopy equivalences that act trivially on $\Ho{d}(W_g)$. The space $\B\mathrm{tor}_\partial(W_{g,1})$ is nilpotent and rationally equivalent to $\trunc 1 {\B\aut_\partial(W_{g,1})}$ and, as discussed in \cref{ex:W_g}, the latter space has Lie model
$$\lie g_g = \trunc 1 {\Der[\omega] (\LL V_g)}.$$
In particular, this means that there is an isomorphism of algebras
$$\Coho * \big(\B\mathrm{tor}_\partial(W_{g,1});\QQ\big) \iso  \Coho[CE]{*} (\lie g_g).$$

The spectral sequence of the homotopy fiber sequence \eqref{eq:hfib} turns out to collapse at the $E_2$-page. This was proved stably (i.e.~for $g$ large compared to the cohomological degree) in \cite{BerglundMadsen20}, and later without any restrictions in \cite[Theorem 4.32]{BerglundZeman21}. This yields an isomorphism (\cite{BerglundZeman21} even shows there is an isomorphism of algebras)
$$
\Coho * (\B \bdryaut(W_{g,1}); \QQ) \iso 
\Coho * \big(\Gamma_g; \Coho[CE]{*} (\lie g_g) \big).
$$
The action of $\Gamma_g$ on the Chevalley--Eilenberg cohomology of $\lie g_g =  \trunc 1 {\Der[\omega] (\LL V_g)}$ is induced by the evident action on $V_g$. By using vanishing results for the cohomology of arithmetic groups due to Borel \cite{Borel81}, one can show that the right-hand side reduces to 
$$
\Coho * (\Gamma_g;\QQ) \tensor \Coho[CE]{*} (\lie g_g)^{\Gamma_g}
$$
in the stable range. In this range, the left factor $\Coho * (\Gamma_g;\QQ)$ is well-understood; results of Borel \cite{Borel75} show that it is a polynomial ring on certain classes $x_i$ for $i\geq 1$.

The right factor $\Coho[CE]{*} (\lie g_g)^{\Gamma_g}$ can be expressed in terms of graph homology, as we now will explain.

Forgetting the homological grading, the Lie algebra $\lie g_g$ is identical to the positive part of the Lie algebra of symplectic derivations considered by Kontsevich \cite{Kontsevich93}. He showed that the Chevalley--Eilenberg cohomology of it is related the homology of the Lie graph complex $\operad F$. 
To account for the homological grading, we define a regraded version $\operad F^d$ of $\operad F$ by
$$\operad F^d \cyc{g,n} = s^{2d(1-g)} \operad F\cyc {g,n}.$$
For a graded vector space $P$, we let
$$\operad F^d[P] = \bigoplus_{g,n} \operad F^d\cyc{g,n} \otimes_{\Sigma_n} P^{\otimes n},$$
and note that
$$\operad F^d[0] = \bigoplus_{g} \operad F^d\cyc{g,0}.$$

The following is essentially dual to \cite[Theorem 9.1]{BerglundMadsen20}.
\begin{theorem}
There is an isomorphism of chain complexes
$$\lim{g} \coCE{*}(\lie g_g)^{\Gamma_g} \iso \SA \operad F^d[0]$$
where the right hand side denotes the free graded commutative algebra on the chain complex $\operad F^d[0]$.
\end{theorem}

The above (together with homological stability results and semisemplicity of the $\Gamma_g$-representation $C^*(\lie g_g)$) lead to 

\begin{theorem}[Berglund--Madsen]
There is an isomorphism
$$
\lim{g} \Coho{*}(\B\bdryaut(W_{g,1});\QQ) \iso \QQ[x_1,x_2,\ldots] \otimes \Lambda \Ho{*}(\operad F^d)[0].
$$
\end{theorem}

\begin{remark}
Stoll \cite{Stoll22} recently proved similar results where $W_g$ is replaced by manifolds of the form
$$\Conn_g \bigl( S^k \times S^\ell \bigr), $$
for $3\leq k < \ell \leq 2k-2$.
An interesting new phenomenon is that a twisted version of the Lie graph complex $\operad F$ appears in the odd dimensional case.
\end{remark}

Berglund--Madsen also consider the block diffeomorphism group $\tDiff_\partial(W_{g,1})$. Interestingly, the full graph complex $\operad F$, and not just the ``vacuum'' part $\operad F[0]$, appears in the description of its stable cohomology. Considerations similar in outline to the above imply the following result, where $P \subset \Coho{*}(BO;\QQ)$ denotes the graded vector space spanned by all Pontryagin classes $p_i$ of degree $4i > d$.
\begin{theorem}[Berglund--Madsen]
There is an isomorphism
$$
\Coho{*}(\B\tDiff_\partial(W_{\infty,1});\QQ) \iso \QQ[x_1,x_2,\ldots] \otimes \Lambda \Ho{*}(\operad F^d)[P].
$$
\end{theorem}

In view of the fact that the Lie graph complex $\operad F$ originated as a device to compute the homology of automorphism groups of free groups \cite{Kontsevich93,Kontsevich94}, it is surprising that it also appears in the cohomology of automorphisms of high dimensional manifolds. Whether a more direct connection between the homology of automorphism groups of free groups and the cohomology of automorphism of high dimensional manifolds can be found is an open problem.

\begin{appendices}
  \crefalias{section}{appsec}
  \crefalias{subsection}{appsec}
  
  \section{Massey products and higher order Whitehead products}

The $C_\infty$-algebra structure on cohomology $\Coho*(X;\QQ)$ and the $L_\infty$-algebra structure on the rational homotopy groups $\pi_{*+1}(X)\otimes \QQ$ are closely related to Massey products and higher order Whitehead products, respectively. The goal of this section is to make these statements more precise and collect some relevant references.

\subsection{Massey products}
The discussion of Massey products does not require commutativity, so let us fix an associative dg algebra $(A,d)$.
The $n$-fold Massey product of $n \ge 2$ homology classes $x_1, \ldots, x_n \in \Ho*(A)$ is a subset
$$\langle x_1, \ldots, x_n \rangle \subseteq \Ho*(A)$$
that is defined whenever $\langle x_i ,\ldots, x_j \rangle$ is defined and contains $0$ for all $1 \le i < j \le n$ such that $(i, j) \neq (1, n)$.
Assume this condition holds; then there exists a family of chains $a_{i,j} \in A$, indexed by all $(i,j) \ne (0,n)$ with $0\leq i < j \leq n$, such that $a_{i-1,i}$ is a representative of $x_i$ and such that
$$d(a_{i,j}) = \sum_{i < k < j} \overline{a}_{i,k} a_{k,j}$$
where $\overline{a} \defeq (-1)^{|a|+1} a$.
Now $\langle x_1, \ldots, x_n \rangle$ is defined to be the set of all homology classes associated to cycles of the form
$$\sum_{0 < k < n} \overline{a}_{0,k} a_{k,n}$$
where $a_{i,j}$ is some family of chains as above.
Note that $\langle x_1, x_2 \rangle$ is given by the singleton $\set{(-1)^{\deg{x_1}+1} x_1 x_2}$.

The following proposition summarizes some relations between Massey products and the $A_\infty$-algebra structure on $\Ho*(A)$ obtained via the homotopy transfer theorem.
When specialized to a cdga model $A$ for a space $X$, this gives relations between the minimal $C_\infty$-algebra structure on $H^*(X;\QQ)$ and Massey products. We refer the reader to \cite{BMFM20} and \cite[\S18.3]{GriffithsMorgan13} for more details.

\begin{proposition} \label{prop:Massey}
  Let $x_1, \dots, x_n \in \Ho*(A)$ and set $e \defeq \sum_{i = 1}^n (n - i) \deg {x_i}$.
  \begin{enumerate}
    \item Assume that $\langle x_1, \ldots, x_n \rangle$ is defined.
    Then for all $x \in \langle x_1, \ldots, x_n \rangle$ we have $[x] = [(-1)^e m_n(x_1, \ldots, x_n)]$ in the quotient $\Ho*(A) / \sum_{k<n} \im(m_k)$.
    \item Assume that $m_k = 0$ for all $k \le n - 2$ and that $\langle x_1,\ldots, x_n \rangle$ is defined.
    Then $(-1)^e m_n(x_1,\ldots,x_n) \in \langle x_1,\ldots, x_n \rangle$.
    \item Each $x \in \langle x_1,\ldots,x_n \rangle$ can be realized as $\epsilon m_n(x_1,\ldots,x_n)$ as above for some sign $\epsilon \in \set{\pm 1}$, after possibly passing to an $\Ainfty$-isomorphic $\Ainfty$-algebra structure on $\Ho*(A)$.
  \end{enumerate}
\end{proposition}


\subsection{Lie--Massey products and higher order Whitehead products}
The analogs of Massey products for dg Lie algebras are called \emph{Lie--Massey products} and they can be defined as follows (cf.~\cite{Retakh93}). For a dg Lie algebra $(L,d)$, the $n$-fold Lie--Massey product of $n \ge 2$ homology classes $\alpha_1, \ldots, \alpha_n \in \Ho*(L)$ is a subset
$$ \langle \alpha_1, \ldots, \alpha_n \rangle \subseteq \Ho*(L) $$
that is defined whenever $\langle \alpha_{i_1}, \ldots, \alpha_{i_k} \rangle$ is defined and contains $0$ for all proper subsets $\{i_1, \ldots, i_k\} \subsetneq \{1, \ldots, n\}$ of cardinality at least $2$.
Assume this condition holds; then there exists a family of chains $u_S \in L$, indexed by all proper subsets $S \subsetneq \{1, \ldots, n\}$, such that $u_{\set i}$ is a representative of $\alpha_i$ and such that
$$ d(u_S) = \sum (-1)^{e_{V,W}} [u_V, u_W] $$
where the sum is over all partitions $S = V \disjunion W$ such that the smallest element of $S$ is contained in $V$.
Here $e_{V,W}$ is defined to be $\sum_{v \in V} (\deg {\alpha_v} + 1)$ plus the Koszul sign incurred by permuting $\eTensor_{s \in S} \shift \alpha_s$ into $\eTensor_{v \in V} \shift \alpha_v \tensor \eTensor_{w \in W} \shift \alpha_w$.
Now $\langle \alpha_1, \ldots, \alpha_n \rangle$ is defined to be the set of all homology classes associated to cycles of the form
$$ \sum (-1)^{e_{V,W}} [u_V, u_W] $$
where $u_S$ is a family as above, and the sum is over all partitions $\set{1, \dots, n} = V \disjunion W$ with $1 \in V$.
Note that $\langle \alpha_1, \alpha_2 \rangle$ is given by the one element set $\set{(-1)^{\deg {\alpha_1} + 1} [\alpha_1, \alpha_2]}$.

The following proposition summarizes some results that connect Lie--Massey brackets to the $L_\infty$-algebra structure on $\Ho*(L)$ obtained from the homotopy transfer theorem, see \cite{BBMFM17}.
Note that there is no analogue of the third statement of \cref{prop:Massey}.

\begin{proposition} \label{prop:whitehead vs Linfty}
  Let $\alpha_1, \dots, \alpha_n \in \Ho*(A)$ and set $e \defeq \sum_{i = 1}^n (n - i) \deg {\alpha_i}$.
  \begin{enumerate}
      \item Assume that $\langle \alpha_1, \ldots, \alpha_n \rangle$ is defined.
    Then, for all $\alpha \in \langle \alpha_1, \ldots, \alpha_n \rangle$, we have $[\alpha] = [(-1)^e l_n(\alpha_1, \ldots, \alpha_n)]$ in the quotient $\Ho*(L) / \sum_{k<n} \im(l_k)$.
    \item Assume that $l_k = 0$ for all $k \le n - 2$ and that $\langle \alpha_1, \ldots, \alpha_n \rangle$ is defined.
    Then $(-1)^e l_n(\alpha_1, \ldots, \alpha_n) \in \langle \alpha_1, \ldots, \alpha_n \rangle$.
  \end{enumerate}
\end{proposition}

The topological significance of Lie--Massey products is that they are related to \emph{higher order Whitehead products}. Higher order Whitehead products are certain partially defined operations on homotopy groups that were introduced in the 60s, cf.~\cite{Hardie61,Porter65}. Allday \cite{Allday73,Allday77} showed that higher order Whitehead products in the rational homotopy groups of a simply connected space $X$ can be computed in terms of Lie--Massey brackets in the homology of a dg Lie model for $X$. Andrews--Arkowitz \cite{AndrewsArkowitz78} computed higher order Whitehead products in terms of Sullivan models. Under the dictionary in \cref{prop:L-infinity dictionary}, this implicitly gives a relation between the higher order Whitehead products and the minimal $\Linfty$-algebra structure on the rational homotopy groups. More explicit statements formulated directly in terms of $\Linfty$-algebras can be found in \cite{BBMFM17}.


  \section{The nerve of an \texorpdfstring{$\Linfty$}{L∞}-algebra}
In this appendix we will review the \emph{Maurer--Cartan space}, or \emph{nerve}, of an $\Linfty$-algebra, and related constructions. The nerve of a positively graded $\Linfty$-algebra of finite type is nothing but the spatial realization of the associated Sullivan algebra, but the advantage of the Maurer--Cartan space perspective becomes apparent when one works with unbounded $\Linfty$-algebras and models for non-connected spaces such as mapping spaces. The nerve also provides a more direct link between spaces and dg Lie algebras than Quillen's theory \cite{Quillen69}.

The nerve of a dg Lie algebra was introduced by Hinich \cite{Hinich97bis} and was further brought into the spotlight by Getzler \cite{Getzler09}. The nerve $\MC_\bullet(\lie g)$, and Getzler's small model for it, $\gamma_\bullet(\lie g)$, have played an important role in recent developments of rational homotopy theory (see e.g.~ \cite{BuijsMurillo13,Berglund15,LazarevMarkl15,BFMT20,Robert-NicoudVallette20}) and we will briefly discuss some of these in this appendix.

\begin{definition} \label{def:nilpotent}
The \emph{lower central series} of an $\Linfty$-algebra $\lie{g}$ is by definition the smallest filtration
$$\lie{g} = \Gamma^1 \lie{g} \supseteq \Gamma^2 \lie{g} \supseteq \ldots$$
that is compatible with the $\Linfty$-structure in the sense that
$$\big[\Gamma^{i_1}\lie{g},\ldots,\Gamma^{i_r}\lie{g}\big] \subseteq \Gamma^{i_1+\ldots+i_r}\lie{g}$$
for all $i_1,\ldots,i_r$.

We say that $\lie{g}$ is \emph{nilpotent} if for every $n$, there exists a $k$ such that
$$\Gamma^k\lie{g}_n = 0.$$
The \emph{completion} of $\lie{g}$ is the $\Linfty$-algebra
$$\widehat{\lie{g}} = \varprojlim \lie{g}/\Gamma^k \lie{g}$$
and we say that $\lie{g}$ is \emph{complete} if the canonical map
$$\lie{g} \longto \widehat{\lie{g}}$$
is an isomorphism.
\end{definition}

\begin{remark}
Clearly, every positively graded $\Linfty$-algebra is nilpotent and every nilpotent $\Linfty$-algebra is complete.
Moreover, a finite type $\Linfty$-algebra is complete if and only if it is nilpotent \cite[Proposition 5.2]{Berglund15}. (Note that the notion of nilpotence considered here coincides with the notion of ``degree-wise nilpotence'' of \cite[Definition 2.1]{Berglund15}.) The definition of completeness given here is slightly more restrictive than the one given in \cite{Berglund15} in that we only consider the lower central series filtration.
\end{remark}

If $\lie{g}$ is a complete $\Linfty$-algebra, then the series
$$\curvature(\tau) = \sum_{n\geq 0} \frac{1}{n!}\ell_n(\tau,\ldots,\tau)$$
converges for every $\tau\in \lie{g}_{-1}$. If
$$\curvature(\tau) = 0,$$
then $\tau$ is called a \emph{Maurer--Cartan element}.
The set of Maurer--Cartan elements in $\lie{g}$ is denoted $\MC(\lie{g})$.

\begin{definition}
If $A$ is a cdga and $\lie{g}$ is a complete $\Linfty$-algebra, the tensor product of the underlying chain complexes $A\tensor{} \lie{g}$ admits an $\Linfty$-algebra structure where
$$[a_1\tensor{} x_1,\ldots,a_n\tensor x_n] = \pm a_1\ldots a_n\tensor{} [x_1,\ldots,x_n]$$
for all $n\geq 2$ and all $a_1,\ldots,a_n\in A$, $x_1,\ldots,x_n\in \lie{g}$. The completion of this $\Linfty $-algebra is called the \emph{completed tensor product} and it is denoted
$$A\ctensor \lie{g}.$$
\end{definition}

This construction is functorial in $A$, so if we form the completed tensor product with the simplicial cdga $\Omega_\bullet$ (see \cref{def:omega}), we obtain a simplicial complete $\Linfty$-algebra $\Omega_\bullet \ctensor \lie{g}$.

\begin{definition}
The \emph{Maurer--Cartan space}, or \emph{nerve}, of a complete $\Linfty$-algebra $\lie{g}$ is the simplicial set
$$\MC_{\bullet}(\lie{g}) = \MC{}(\Omega_\bullet \ctensor \lie{g}).$$
\end{definition}

The vertices of the simplicial set $\MC_\bullet(\lie{g})$ are the Maurer--Cartan elements of $\lie{g}$ and the connected components are the so-called gauge equivalence classes of Maurer--Cartan elements,
$$\pi_0(\MC_\bullet(\lie{g})) = \MC(\lie{g})/{\sim}.$$
The higher homotopy groups can be computed in terms of the homology of twists of $\lie{g}$ by Maurer--Cartan elements. For a Maurer--Cartan element $\tau\in\MC(\lie{g})$, the twisted $\Linfty$-algebra $\lie{g}^\tau$ is the complete $\Linfty$-algebra with the same underlying graded vector space as $\lie{g}$ and structure maps $\ell_n^\tau$ given by
$$\ell_n^\tau(x_1,\ldots,x_n) = \sum_{k\geq 0} \frac{1}{k!} \ell_{k+n}(\tau,\ldots,\tau,x_1,\ldots,x_n).$$
The following is \cite[Theorem~1.1]{Berglund15}.

\begin{theorem} \label{thm:homotopy groups of nerve}
For every Maurer--Cartan element $\tau$ in $\lie{g}$ and every $k\geq 0$, there is an isomorphism
$$\pi_{k+1}(\MC_\bullet(\lie{g}),\tau) \cong \Ho{k}(\lie{g}^\tau).$$
For $k>0$, this is an isomorphism of abelian groups.
For $k=0$, this is an isomorphism of groups where the complete Lie algebra $\Ho{0}(\lie{g}^\tau)$ is given a group structure via the Baker--Campbell--Hausdorff formula.
\end{theorem}


Under the correspondence in \cref{prop:L-infinity dictionary}, the Maurer--Cartan space corresponds to Sullivan's simplicial realization.
The following is \cite[Proposition 6.1]{Berglund15}.

\begin{proposition} \label{prop:mc vs spatial}
Let $L$ be a non-negatively graded nilpotent $\Linfty$-algebra of finite type and let
$$\coCE{*}(L) = (\Lambda V,d)$$
be the corresponding Sullivan algebra.
The Maurer--Cartan space of $L$ is isomorphic to the simplicial realization of $(\Lambda V,d)$,
$$\MC_\bullet(L) \cong \langle \Lambda V,d \rangle.$$
\end{proposition}

\begin{remark} \label{remark:homotopy groups}
The graded vector space $V$ is dual to $sL$, so \cref{prop:mc vs spatial} together with Theorem \ref{thm:homotopy groups of nerve} can be used to prove the fact (Proposition \ref{prop:homotopy groups from sullivan model}) that the rational homotopy groups of a simply connected space are dual to the indecomposables of the minimal Sullivan model,
$$\pi_k \big( \langle V, d \rangle \big) \cong \Hom(V^k,\QQ)$$
for all $k$.
\end{remark}

Let us give an example of an application of Maurer--Cartan spaces to rational homotopy theory. They are very useful for studying mapping spaces. The following is \cite[Theorem 1.4]{Berglund15}.

\begin{theorem}
If $A$ is a cdga model for a connected space $X$ and $L$ is a complete $\Linfty$-model for a $\QQ$-local nilpotent space $Y$, then $A\ctensor L$ is a model for the space of maps from $X$ to $Y$ in the sense that
$$\map(X,Y) \eq \MC_\bullet(A\ctensor L).$$
In particular, there is a bijection
\begin{equation} \label{eq:mapping space bijection}
[X,Y] \cong \MC(A\ctensor L)/{\sim}
\end{equation}
and for all $k\geq 0$ there is an isomorphism
$$\htpygrp{k+1}\big(\map(X,Y),f\big) \cong \Ho{k}\big( (A\ctensor L)^\tau \big),$$
whenever the homotopy class of the map $f\colon X\to Y$ corresponds to the gauge equivalence class of the Maurer--Cartan element $\tau\in \MC(A\ctensor L)$ under the bijection \eqref{eq:mapping space bijection}.
\end{theorem}

\begin{remark}
While it is in principle possible to write down models for mapping spaces in the classical language of Sullivan algebras---in fact this has been done by Brown--Szczarba \cite{BrownSzczarba97})---the model formulated using the language of $\Linfty$-algebras is arguably more perspicuous.
\end{remark}

We end this appendix with a brief review of Getzler's functor $\gamma_\bullet(\lie{g})$, introduced in \cite{Getzler09}, and some recent alternative constructions of it due to Buijs--F\'elix--Murillo--Tanr\'e \cite{BFMT20} and Robert-Nicoud--Vallette \cite{Robert-NicoudVallette20}.

Using Dupont's contraction \eqref{eq:dupont} specialized to the $n$-simplex $X= \Delta^n$, Getzler defines
$$\gamma_n(\lie{g}) = \MC_n(\lie{g}) \cap \ker (s \tensor 1)$$
and shows that this defines a simplicial subset $\gamma_\bullet(\lie g)$ of $\MC_\bullet(\lie g)$ that has a number of remarkable properties, including:
\begin{itemize}
\item $\gamma_\bullet(\lie{g})$ is a Kan complex and the inclusion $\gamma_\bullet(\lie{g}) \to \MC_\bullet(\lie{g})$ is a homotopy equivalence.

\item For $\lie{g}$ abelian, $\gamma_\bullet(\lie{g})$ is isomorphic to the simplicial vector space associated to the chain complex $(\shift \lie{g})\langle 0 \rangle$ under the Dold--Kan correspondence.

\item For a nilpotent Lie algebra $\lie g$ concentrated in degree zero, $\gamma_\bullet(\lie{g})$ is isomorphic to the nerve of the nilpotent group $\exp( \lie{g})$ associated to $\lie g$, i.e., the group with underlying set $\lie{g}$ and multiplication given by the Baker--Campbell--Hausdorff formula.

\end{itemize}




Recently, alternative constructions of Getzler's functor have been found. This is a nice application of the technology discussed in the first lecture, so let us sketch the main ideas. Buijs--F\'elix--Murillo--Tanr\'e observe that one can use the natural $C_\infty$-algebra structure on $\Cochains{*}(X;\QQ)$ (see \cref{remark:dupont}) to define a cosimplicial complete dg Lie algebra
$$\bfmt^\bullet =\coHarr{*}\big(\Cochains{*}(\Delta^\bullet;\QQ)\big).$$
Buijs--F\'elix--Murillo--Tanr\'e \cite{BFMT20} and Robert-Nicoud \cite{Robert-Nicoud19} proved that this cosimplicial object represents Getzler's functor.
\begin{theorem}[]
There is a natural isomorphism
$$\gamma_\bullet(\lie{g}) \cong \Hom[\cDGL](\bfmt^\bullet,\lie{g})$$
for complete dg Lie algebras $\lie{g}$.
\end{theorem}

In particular, this means that the functor $\gamma_\bullet$ admits a well-behaved left adjoint when restricted to complete dg Lie algebras; the cosimplicial complete dg Lie algebra $\bfmt^\bullet$ gives rise to an adjunction
\begin{equation} \label{eq:bfmt adjunction}
\begin{tikzcd}
  \sSet \rar[bend left]{\bfmt(\blank)}[swap, name = U]{} & \cDGL \lar[bend left]{\Realization{\blank}}[swap, name = D]{}
  \ar[from = U, to = D, phantom]{}{\vertdashv}
\end{tikzcd}
\end{equation}
where
\begin{align*}
  \bfmt(X) &\defeq \colim{\Simplex {n} \to X} \bfmt^n \\
  \Realization{\lie{g}} &\defeq \Hom[\cDGL](\bfmt^\bullet,\lie{g}).
\end{align*}
The homotopy theoretical properties of this adjunction are worked out in the recent book \cite{BFMT20}.

Next, one might wonder what happens if we do not restrict $\gamma_\bullet$ to complete dg Lie algebras. Is the functor $\gamma_\bullet$ representable by a cosimplicial complete $\Linfty$-algebra?
Recently, Robert-Nicoud--Vallette \cite{Robert-NicoudVallette20} gave a positive answer to this question. The key is to work with the canonical resolution of the commutative operad,
\begin{equation} \label{eq:canonical}
\Cobar \Bar \Com \xlongto{\eq} \Com,
\end{equation}
instead of the minimal resolution
$$\Com_\infty = \Cobar \dual{(\opshift \Lie)} \xlongto{\eq} \Com.$$
Note that the cooperad $\Bar \Com$ may be identified with the dual of the shifted $\Linfty$-operad $\opshift \Lie_\infty$. This means in particular that one can associate a complete $\Linfty$-algebra,
$$\coHarrinfty{*}(A),$$
to every $\Cobar \Bar \Com$-algebra $A$. This could be viewed as an $\Linfty$-version of Harrison cochains. Viewing $\PDR(X)$ as a $\Cobar \Bar \Com$-algebra by pullback along \eqref{eq:canonical}, an application of the homotopy transfer theorem to Dupont's contraction \eqref{eq:dupont} shows that there is a natural $\Cobar \Bar \Com$-algebra structure on $\Cochains{*}(X;\QQ)$ for every simplicial set $X$. In particular, there is a cosimplicial complete $\Linfty$-algebra
$$\rnv^\bullet = \coHarrinfty{*}\big(\Cochains{*}(\Delta^\bullet;\QQ)\big).$$
Robert-Nicoud--Vallette \cite{Robert-NicoudVallette20} then prove the following.
\begin{theorem}
There is a natural isomorphism
$$\gamma_\bullet(\lie{g}) \cong \Hom[\cLinfty](\rnv^\bullet,\lie{g}),$$
for complete $\Linfty$-algebras $\lie{g}$.
\end{theorem}
As before, a consequence of this result is that Getzler's functor admits a left adjoint that can be described in quite concrete terms; there is an adjunction
\begin{equation} \label{eq:rnv adjunction}
\begin{tikzcd}
  \sSet \rar[bend left]{\rnv(\blank)}[swap, name = U]{} & \cLinfty \lar[bend left][below]{\Realization{\blank}_\infty}[swap, name = D]{}
  \ar[from = U, to = D, phantom]{}{\vertdashv}
\end{tikzcd}
\end{equation}
where
\begin{align*}
  \rnv(X) &\defeq \colim{\Simplex {n} \to X} \rnv^n, \\
  \Realization{\lie{g}}_\infty &\defeq \Hom[\cLinfty](\rnv^\bullet,\lie{g}).
\end{align*}
We refer to \cite{Robert-NicoudVallette20} for a detailed study of this construction. The cosimplicial objects $\bfmt^\bullet$ and $\rnv^\bullet$, and the adjunctions they give rise to, will likely play an important role in future developments of rational homotopy theory.
  \section{Rational homotopy theory for non-nilpotent spaces} \label{sec:nonnilpotent rht}
In this appendix we summarize some aspects of rational homotopy theory for spaces that are not necessarily nilpotent. We discuss Bousfield's $\Ho{*}(-;\QQ)$-localization and its relation to the Bousfield--Kan $\QQ$-completion and the Malcev completion for groups, and we discuss what Sullivan's minimal model can say about a non-nilpotent space.

Bousfield's $\Ho*(-;\QQ)$-localization \cite{Bousfield75} is a natural generalization of the rationalization of a nilpotent space (see e.g.~\cite{HiltonMislinRoitberg75} or \cite[Chapter 9]{FelixHalperinThomas01} for accounts of the latter). The $\Ho*(-;\QQ)$-localization is a functor $\HQloc$ from the homotopy category of pointed spaces to itself, equipped with a natural transformation $\eta_X \colon X\to \HQloc(X)$.
It is characterized by the following properties:
\begin{enumerate}
\item The map $\eta_X\colon X\to \HQloc(X)$ is a rational homology isomorphism.
\item Every rational homology isomorphism $f\colon X\to Y$ factors uniquely in the homotopy category as $f = r\eta_X$ for some $r\colon \HQloc(X) \to Y$.
\end{enumerate}


The Bousfield--Kan $\QQ$-completion $X\to \BKQcompl X$ \cite{BousfieldKan72} provides a concrete model for the $\Ho*(-;\QQ)$-localization in many, but not all, cases. We summarize some of its properties (see \cite{BousfieldKan72}):



\begin{itemize}
\item A map $X\to Y$ is a rational homology isomorphism if and only if the induced map $\BKQcompl X\to \BKQcompl Y$ is a weak homotopy equivalence.

\item The canonical map $X\to \BKQcompl X$ may or may not be a rational homology isomorphism. The space $X$ is called \emph{$\QQ$-good} if it is and \emph{$\QQ$-bad} if it is not.
\item If $X$ is $\QQ$-good, then the Bousfield--Kan $\QQ$-completion $\BKQcompl X$ is a model for the $\Ho{*}(-;\QQ)$-localization.

\item Nilpotent spaces, and more generally virtually nilpotent spaces, are $\QQ$-good \cite[Proposition 3.4]{DrorDwyerKan77}. In particular, simply connected spaces or spaces with finite fundamental group are $\QQ$-good.

\item If $X$ is $\QQ$-bad, then so is $\BKQcompl X$.
An example of a $\QQ$-bad space is the wedge of $n$ circles $\vee^n S^1$ for $n>1$ \cite{IvanovMikhailov19}.
\end{itemize}

The following key result links the Bousfield--Kan completion to the minimal Sullivan model.
\begin{theorem}
Consider a connected simplicial set $X$ of finite $\QQ$-type, and let $(\Lambda V,d) \to \Omega^*(X)$ be the minimal Sullivan model. The adjoint map
$$X \longto \langle \Lambda V, d \rangle$$
is weakly equivalent to the Bousfield--Kan $\QQ$-completion $X\to \BKQcompl X$.
Furthermore, there is a bijection for every $k\geq 1$
$$\htpygrp{k}(\BKQcompl X) \cong \Hom(V^k,\QQ).$$
For $k\geq 2$ this is an isomorphism of abelian groups.

\end{theorem}

\begin{proof}
See Theorem \cite[Theorem 12.2]{BousfieldGugenheim76} and \cite[Theorem 12.8(iii)]{BousfieldGugenheim76}.
\end{proof}


If $X$ is nilpotent, then there is a simple relation between the homotopy groups of $X$ and the homotopy groups of $\BKQcompl X$, namely
$$\htpygrp{k}(\BKQcompl X) \iso \htpygrp{k}(X)\otimes \QQ$$
for all $k$, where $\htpygrp{1}(X)\otimes \QQ$ should be interpreted as the rationalization of the nilpotent group $\htpygrp{1}(X)$.

For general $X$, the relation is not as simple, but some things can be said.
If $X$ is connected and $\Ho{1}(X;\QQ)$ is finite-dimensional, then $\htpygrp{1}(\QQ_\infty X)$ may be identified with the Malcev completion $\htpygrp{1}(X)_\QQ^\wedge$ (cf.~\cite[Corollary 7.4]{FelixHalperinThomas15}).

The Malcev completion of a group $G$ can be defined by
$$G_\QQ^\wedge = \varprojlim\, \big((G/\Gamma_n G) \otimes \QQ\big),$$
i.e.\ the inverse limit of the rationalizations of the nilpotent groups $G/\Gamma_n G$, where $\{ \Gamma_n G\}_n$ is the lower central series of $G$, defined by $\Gamma_1 G = G$ and $\Gamma_{n+1}G = [G,\Gamma_nG]$.
Alternatively, one can define $G_\QQ^\wedge$ as the grouplike elements in the complete Hopf algebra $\widehat{\QQ}[G]$, see \cite[Appendix A3]{Quillen69}.

In general, the homotopy groups of $\BKQcompl X$ can be wildly different from the rational homotopy groups of $X$. For example, if $X = \B\Gamma$ for a perfect group $\Gamma$ of finite type, then the higher homotopy groups of $X$ are trivial, whereas $\BKQcompl X$ is a model for the rationalization of the Quillen plus construction $\B\Gamma^+$, so that
$$\pi_k(\BKQcompl \B\Gamma) \cong \pi_k(\B\Gamma^+)\otimes \QQ$$
for all $k$, cf.~\cite[pp.\ 212f.]{FelixHalperinThomas01}.
\end{appendices}

\bibliographystyle{alpha}
\bibliography{bib_global_Berglund_Stoll}

\end{document}